\documentclass[12pt]{amsart}
\usepackage{amsmath}
\usepackage{amsfonts}
\usepackage{amssymb}
\usepackage{mathtools}
\usepackage{enumitem}
\usepackage{cite}
\usepackage[alphabetic]{amsrefs}
\usepackage{mathrsfs}

\usepackage[usenames,dvipsnames]{xcolor}

\usepackage[english]{babel}
\usepackage[latin1]{inputenc}
\usepackage{fancyhdr}
\usepackage{indentfirst}
\usepackage{graphicx}
\usepackage{float}
\usepackage{caption}
\usepackage{newlfont} 
\usepackage{amssymb}
\usepackage{amsmath}
\usepackage{mathtools}
\usepackage{latexsym}
\usepackage{amsthm}
\usepackage{enumitem}
\usepackage{esint}
\theoremstyle{plain}
\newtheorem{thm}{Theorem}[section]
\newtheorem{exp-thm}[thm]{Expected Theorem}
\newtheorem{exp-lem}[thm]{Expected Lemma}
\newtheorem{exp-cor}[thm]{Expected Corollary}
\newtheorem{lem}[thm]{Lemma}
\newtheorem{prop}[thm]{Proposition}    
\newtheorem{cor}[thm]{Corollary}
\newtheorem{defin}[thm]{Definition}
\theoremstyle{remark}                 
\newtheorem{remark}[thm]{Remark}

\newcommand{\R}{\mathbb{R}}

\newcommand{\N}{\mathbb{N}}

\newcommand{\G}{\mathbb{G}}

\newcommand{\scal}[2]{\langle {#1} , {#2}\rangle}
\newcommand{\Hnabla}{\nabla_{\G}}
\newcommand{\divg}{\mathrm{div}_{\G}}
\newcommand{\bq}{\mathbf{q}}
\newcommand{\bxi}{\boldsymbol{\xi}}
\newcommand{\e}{\varepsilon}
\renewcommand{\leq}{\leqslant}
\renewcommand{\le}{\leqslant}
\renewcommand{\geq}{\geqslant}
\renewcommand{\ge}{\geqslant}
\allowdisplaybreaks

\makeatletter

\def\cleardoublepage{\clearpage\if@twoside \ifodd\c@page\else
	\hbox{}
	\thispagestyle{empty}
	\newpage
	\if@twocolumn\hbox{}\newpage\fi\fi\fi}
\makeatother

\numberwithin{equation}{section}

\begin{document}	
	\title{Regularity for almost minimizers of a one-phase Bernoulli-type functional in Carnot Groups of step two}
	
	\author{Fausto Ferrari}
	\address{Fausto Ferrari: Dipartimento di Matematica Universit\`a di Bologna Piazza di Porta S.Donato 5  40126, Bologna-Italy}
	\email{fausto.ferrari@unibo.it }
	
	\author{Nicol\`o Forcillo}
	\address{Nicol\`o Forcillo: Department of Mathematics, Michigan State University,
619 Red Cedar Road, East Lansing, MI 48824, USA}
	\email{forcill1@msu.edu}

	\author{Enzo Maria Merlino}
	\address{Enzo Maria Merlino: Dipartimento di Matematica Universit\`a di Bologna Piazza di Porta S.Donato 5 40126, Bologna-Italy}
	\email{enzomaria.merlino2@unibo.it }
	
	\thanks{The authors are members of Gruppo Nazionale per l'Analisi Matematica, la Probabilit\'a e le loro Applicazioni (GNAMPA) of the Istituto Nazionale di Alta Matematica (INdAM): F.F. has been partially supported by the GNAMPA research project 2023 "Equazioni completamente non lineari locali e non locali", N.F. has been supported by the GNAMPA research project 2023 "Problemi variazionali/nonvariazionali: interazione tra metodi integrali e principi del massimo" and E.M.M. has partially supported by the GNAMPA research project 2023 "Equazioni nonlocali di tipo misto e geometrico". F.F. and E.M.M. have been supported by the PRIN research project 2022 7HX33Z - CUP J53D23003610006, "Pattern formation in nonlinear phenomena".  The authors wish to thank  Bruno Franchi and Francesca Corni for some useful discussions.}
        \date{\today}
	
	\begin{abstract}
		We prove that nonnegative almost minimizers of the horizontal Bernoulli-type functional
		$$		J(u,\Omega):=\int_{\Omega}\Big(|\Hnabla u(x)|^2+\chi_{\{u>0\}}(x)\Big)\,dx$$
		are Lipschitz continuous with respect to the Carnot-Carath\'eodory distance.
	\end{abstract}
        \maketitle
	\tableofcontents
	
	\section{Introduction and main result}

	In this paper, we study the regularity of almost minimizers of Bernoulli-type energy functionals in Carnot groups. 
	
	In the Euclidean setting, the regularity of minimizers for the classical one-phase Bernoulli energy functional
\begin{equation}\label{eq-intro}		
\mathscr{F}(u,\Omega):=\int_{\Omega}\Big(|\nabla u(x)|^2+\chi_{\{u>0\}}(x)\Big)\,dx
\end{equation}
    has been deeply studied since the pioneering work of Alt and Caffarelli, see \cite{AC}. We refer the reader to \cites{Vel} and the references therein for a comprehensive description of the subject.

	More recently, the regularity of almost minimizers associated with Bernoulli-type functionals was investigated as well. We recall that the notion of \emph{almost minimality} is a relaxed notion of minimality that arises from variational problems with constraints, such as problems with prescribed volume or curvature and obstacle problems. Almgren introduced the notion of almost minimizers in geometric measure theory \cite{Alm76}. In the nonparametric setting, this notion was introduced by Anzellotti in \cite{Anz83}. Roughly speaking, almost minimizers may be considered as local perturbations of minimizers, which have smaller contribution at small scales. Hence, one may expect that similar regularity results for minimizers hold also for almost minimizers. Nevertheless, the condition satisfied by almost minimizers does not allow to obtain an Euler-Lagrange equation as in the case of minimizers. 
    
    Concerning Bernoulli-type functionals, in \cite{DT}, David and Toro proved that almost minimizers of \eqref{eq-intro} are locally Lipschitz continuous, which is the optimal regularity for minimizers. Results about the free boundary regularity can be found in \cite{DET}. These works employ variational techniques involving tools coming from potential theory and geometric measure theory. 
    We refer to \cite{vega} for a comprehensive overview on the topic.
    More recently, De Silva and Savin, in \cite{DS}, introduced a more direct approach to obtain Lipschitz regularity of almost minimizers of \eqref{eq-intro} based on a dichotomy argument.
Since their techniques involve, mainly, only metric properties, we point out that they may be extended to different more general frameworks. In this direction, see, for instance, \cite{DS-thin} and \cite{DFFV}, dedicated to the \emph{thin} Bernoulli functional and a one-phase problem driven by the $p$-Laplace operator respectively.

	Free boundary problems arising from the minimization of Bernoulli-type functionals may be formulated even in the noncommutative setting of Carnot groups. See, for instance, \cite{DF} for the case of the Heisenberg group. Carnot groups are an intensively-studied noncommutative structure playing a relevant role in many applications. For instance, they represent the tangent model of a general sub-Riemannian manifold and are the natural framework to describe some systems with nonholonomic constraints. On Carnot groups, there exists a comprehensive literature involving geometric measure theory \cites{Mon,FSSC01,FSSC03}, subelliptic partial differential equations \cites{BLU, Fo, FS}, and differential geometry \cites{LD15,LD16,CDPT}. Nevertheless, very little is known on Bernoulli free boundary problems in this setting. For the one-phase case, in \cite{FV}, only in the Heisenberg group, the authors proved the  existence of minimizers, which turn out to be intrinsic harmonic away from the free boundary and present linear growth close to the free boundary. Besides, they proved some density estimates. Furthermore, we recall that, concerning the two-phase case, classical tools seem to fail in this setting. For instance, as recently shown in \cites{FF,FG}, an intrinsic Alt-Caffarelli-Friedman monotonicity formula, written as the natural counterpart to the classical Euclidean one, does not hold. In addition, the viscosity approach occurs some difficulties due to the characteristic points. For instance, this is the case of the comparison principle, see \cite{DF}*{Section 8}. Moreover, we point out that, for these functionals, the regularity results do not follow from the classical theory in the calculus of variations, since the integrand of the functional is discontinuous (in the classical case, the integrand is required to be, for example, Lipschitz continuous, see e.g., \cite{G}*{condition (8.48)} in the Euclidean context or, for example, under stronger regularity assumption in \cites{CG-jems,MZ}, in Carnot groups).

 The setting in which we are working on may be summarized as follows. Let $\G$ be a Carnot group of step two. Assume that $\Omega\subset\G$ is a measurable domain. We define the following energy functional
	\begin{equation}\label{def-J}
		J(u,\Omega):=\int_{\Omega}\Big(|\Hnabla u(x)|^2+\chi_{\{u>0\}}(x)\Big)\,dx,
	\end{equation}
where $u$ belongs to the \emph{horizontal Sobolev space} $HW^{1,2}(\Omega)$, $u\geq0$ a.e. in $\Omega$, $\Hnabla u$ is the so-called \emph{horizontal gradient} of $u$ and $dx$ denotes the Haar measure of the group. We refer to Section \ref{sec:preliminary} for the main notation and definitions.

In the present paper, we focus on some regularity properties of almost minimizers for $J$.

\begin{defin}\label{def-AM} Let $\kappa\ge0$ and $\beta>0$, and $\Omega\subseteq\G$ be an open set. We say that $u\in HW^{1,2}(\Omega)$ is an almost minimizer of $J$ in $\Omega$, with constant $\kappa$ and exponent $\beta$, if $u\geq 0$ $\mathcal{L}^n$-almost everywhere in $\Omega$ and
		\begin{equation}\label{inequality-of-J-p-u-alm-min}
		J(u,B_\varrho(x))\leq(1+\kappa \varrho^\beta)J(v,B_\varrho(x)),
		\end{equation}
		for every metric ball $B_\varrho(x)$, of radius $\varrho\leq0$ and center $x\in\G$, such that $\overline{B_\varrho(x)}\subset \Omega$ and for every $v\in HW^{1,2}(B_\varrho(x))$ such that $u-v\in HW^{1,2}_0(B_\varrho(x))$.
	\end{defin}

    Our main result is the following one.
	\begin{thm}\label{theor-Lipsch-contin-alm-minim}
		Let $u$ be an almost minimizer for  $J$ in $B_1$ with constant $\kappa$ and exponent $\beta$. Then,
		\[\left\|\Hnabla u\right\|_{L^{\infty}(B_{1/2})}\le C\Big(\left\|\Hnabla u\right\|_{L^2(B_1)}+1\Big)\]
        where  $C>0$ is a constant only depending on the homogeneous dimension $Q$, $\kappa,$ and  $\beta$.
		
		In addition, $u$ is uniformly Lipschitz continuous in a neighborhood of  $\left\{u=0\right\}$, namely if  $u(0)=0,$ then
		\[\left|\Hnabla u\right|\le C\quad \mbox{in }B_{r_0},\]
		for some  $C>0$, only depending on  $Q$, $\kappa$ and  $\beta$, and  $r_0\in(0,1)$, depending on  $Q$, $\kappa$, $\beta$ and  $\|\Hnabla u\|_{L^2(B_1)}$.
	\end{thm} 
	In particular, the main contribution of our paper concerns the Lipschitz intrinsic regularity of almost minimizers. We remark that our result is also relevant for the case of minimizers (when $\kappa=0$), since it is new also in this direction. Our approach is a generalization of the method  introduced in \cite{DS}. Nevertheless, some tools in \cite{DS}, mainly based on comparison arguments with harmonic replacements, cannot be applied on Carnot groups. Indeed, if the step of the group is strictly grater then one, the vector fields $X_1,\ldots, X_{m_1}$, generating the first layer of the Lie algebra associated to $\G$,  have no trivial commutators, i.e.
$$
\left[X_i,X_j\right]\not=0 \quad \Leftrightarrow \quad X_iX_j\not=X_jX_i, \quad\text{for every } i,j\in \{1,\ldots,m_1\}, \text{ if } i\not=j.
$$
Consequently, denoting by $\mathcal{L}=-\sum_{j=1}^{m_1}X^2_j$ the positive sub-Laplacian on $\G$, it results
\begin{equation}\label{noncommuta_intro}
X_i\mathcal{L}\not=\mathcal{L}X_i
\end{equation} for all $i=1,\ldots,m_1$. The lack of commutativity, encoded in \eqref{noncommuta_intro}, implies that functions $X_iu$, for $i\in \{1,\ldots,m_1\}$, are not $\G$-harmonic even if $u$ is. Therefore, the mean value characterization of $|\Hnabla u|^2$ on which the argument proposed in \cite{DS} is based, is no longer applicable in our case. To face this aspect, we apply the regularity estimates for subelliptic equations, proved in \cite{CM22}, see also \cite{SS} in the case of the Heisenberg group. In addition, following the original idea of Anzellotti in \cite{Anz83}, we exploit well-known results concerning the Morrey-Campanato-type estimates, developed in \cite{MS79} in  homogeneous spaces. Moreover, we employ some classical tools coming from geometric analysis in Carnot groups, such as  Poincar\'e and Sobolev-Poincar\'e inequalities, see, for instance, \cites{Jer,FLW95,FLW96,Lu92,Lu94}. Finally, it is worth recalling that our Lipschitz intrinsic regularity result only implies  H\"older continuous regularity from a Euclidean point of view, since the Carnot-Charath\'eodory distance is not equivalent to the Euclidean one, see, e.g., \cite{BLU}*{Proposition 5.15.1}.

	The classical motivations to study Bernoulli free boundary problems come from flows with jets and cavities, see, for instance, \cite{CSbook}*{Section 1.1}, heat flows \cite{Ack77}, electrochemical machining \cite{LS87}, combustion theory \cite{BCN}, electrical impedance tomography \cite{AP98}, and phase transitions, see, e.g., \cites{PV1,PV2,Val06}.

	The paper is organized as follows. In Section \ref {sec:preliminary}, we recall basic notions concerning Carnot groups and some useful tools. Next, in Section \ref{sec:dichotomy}, we deal with a dichotomy property, a key step to prove our main result. Finally, Section \ref{sec:lip1} contains the proof of Theorem \ref{theor-Lipsch-contin-alm-minim}. We remind, even if it is not explicitly stated, that all the constants appearing in the sequel may change from line to line when they are universal, i.e. when they depend only on the group $\G$.

		\section{Notation and preliminary results}\label{sec:preliminary}
	
	\subsection{Carnot Groups}
	
	A connected and simply connected Lie group $(\G,\circ)$  (in general non-abelian) is said to be a  {\em Carnot group  of step $k$} if its Lie algebra ${\mathfrak{g}}$ admits a {\it $k$-step stratification}, i.e. there exist linear subspaces $V_1,\ldots,V_k$ such that
	\begin{equation}\label{stratificazione}
		{\mathfrak{g}}=V_1\oplus\ldots\oplus V_k,\quad [V_1,V_i]=V_{i+1},\quad
		V_k\neq\{0\},\quad V_i=\{0\}{\textrm{ if }} i>k,
	\end{equation}
	where $[V_1,V_i]$ is the subspace of ${\mathfrak{g}}$ generated by the commutators $[X,Y]$ with $X\in V_1$ and $Y\in V_i$. The first layer $V_1$, the so-called {\sl horizontal layer},  plays a key role in the theory since it generates all the algebra $\mathfrak g$ by commutation. We point out that a stratified Lie algebra can admit more than one stratification. However, the stratification turns out to be unique up to isomorphisms, thus, the related Carnot group structure is essentially unique (see \cite{LD16}*{Proposition 1.17}). Note that when $k=1$ the group $\G$ is abelian, so we return to the Euclidean situation.

	Setting $m_i=\dim(V_i)$ and $h_i=m_1+\dots +m_i$, with $h_0=0$, for $i=1,\dots,k$ (so that $h_k=n=\dim \mathfrak g=\dim \G$), we choose a basis $e_1,\dots,e_n$ of $\mathfrak g$ adapted to the stratification, that is such that
	$$e_{h_{j-1}+1},\ldots,e_{h_j}\;\text {is a basis of}\; V_j\;\text{for each}\; j=1,\ldots, k.$$
    Let $X=\{X_1,\dots,X_{n}\}$ be the family of left-invariant vector fields such that $X_i(e)=e_i$, $i=1,\dots,n$, where $e$ is the identity of $(\G,\circ)$. Thanks to \eqref{stratificazione}, the subfamily $\{X_1,\dots,X_{m_1}\}$ generates by commutations all the other vector fields, we will refer to $X_1,\ldots,X_{m_1}$ as {\it generating vector fields} of $(\G,\circ)$.  
	
	The map $X\mapsto X(e)$, associating a left-invariant vector field $X$ to its value in $e$, defines an isomorphism from $\mathfrak g$ to $T\G_e$ (in turn identified with $\R^n$). We systematically use these identifications. Furthermore, by the assumption that $\G$ is simply connected, the exponential map is a global real-analytic diffeomorphism from $\mathfrak g$ onto $\G$ (see, e.g., \cites{V,CGr}), so each $x\in\G$ can be written in a unique way as $x=\exp(x_1X_1+\dots+x_nX_n)$. Using these {\it exponential coordinates}, we identify $x$ with the $n$-tuple $(x_1,\ldots,x_n)\in\R^n$ and $\G$ with $(\R^n,\ast)$ where the explicit expression of the group operation $\ast$ is determined by the Campbell-Hausdorff formula (see, for example, \cites{BLU,FS}). In these coordinates, $e=(0,\dots,0)$ and $(x_1,\dots,x_n)^{-1}=(-x_1,\dots,-x_n)$, and the adjoint operator in $L^2(\G)$ of $X_j$, $X_j^{*}$, for $j=1,\dots,n$, turns out to be $-X_j$ (see, for instance,\cite{FSSC03}*{Proposition 2.2}). Moreover, if $x\in \G$ and $i=1,\ldots,k$, we set $x^{(i)}:= (x_{h_{i-1}+1},\ldots,x_{h_{i}})\in \R^{m_i}$, so that we can also identify $x$ with $[x^{(1)},\ldots,x^{(k)}]\in \R^{m_1}\times \ldots \times \R^{m_k}=\R^n$.
	
		Two important families of automorphism of $\G$ are the so-called  \emph{intrinsic translations} and \emph{dilations}. For any $x\in\G$, the {\it (left) translation} $\tau_x:\G\to\G$ is defined as 
	$$ z\mapsto\tau_x z:= x\circ z. $$
For any $ \lambda> 0 $ we call \emph {dilation}, associated with $ \lambda $ the linear automorphism of $\mathfrak{g}$ such that $$\delta_\lambda(v)=\lambda^iv, \quad \text{ if } v\in V_i,\ \text{  for }i=1,\ldots,k.$$ Through the exponential map, one can transfer the notion of dilation on the group $\G$. We denote again by $\delta_\lambda$ the map 
\begin{equation*}
\exp\circ\;\delta_\lambda\circ\exp^{-1}: \G\rightarrow\G.
\end{equation*}
Using the identification given by exponential coordinates, one can read the map $ \delta_ \lambda: \mathbb {G} \rightarrow \mathbb {G} $ as   
	\begin{equation}\label{dilatazioni}
		\delta_\lambda(x_1,\ldots,x_n)=(\lambda^{\alpha_1}x_1,\ldots,\lambda^{\alpha_n}x_n),
	\end{equation} 
	where $\alpha_i\in\N$ is called {\it homogeneity of the variable} $x_i$ in $\G$ (see  \cite{FS}*{Chapter 1}) and is defined as
	\begin{equation*}
		\alpha_j=i \quad\text {whenever}\; h_{i-1}+1\leq j\leq h_{i},
	\end{equation*}
	hence $1=\alpha_1=\ldots=\alpha_m<\alpha_{m+1}=2\leq\ldots\leq\alpha_n=k.$
	
	From Definition \eqref{dilatazioni}, one can easily verify the following properties of intrinsic dilations.
	\begin{lem}\label{DilationsProperties}
		For all  $\lambda, \mu >0$, one has:
		\begin{enumerate}
			\item[(1)]  $\delta_{1} = \mathrm{id}_\G$;\\
			\item[(2)] $\delta^{-1}_{\lambda} = \delta_{\lambda^{-1}}$;\\
			\item[(3)]
			$\delta_{\lambda}\circ \delta_{\mu} =\delta_{\lambda \mu}$;\\
			\item[(4)] for every $p, p'\in \G$, it holds $\delta_\lambda(p) \cdot \delta_\lambda(p') = \delta_\lambda(p\cdot p')$.
		\end{enumerate}
	\end{lem}

	By left translation, the horizontal layer determines a subbundle of the tangent bundle $T\G$ over $\G$. This subbundle, spanned by the family of vector fields $X=(X_1,\dots,X_{m_1})$, plays a crucial role in the theory and is called the {\it horizontal bundle} $H\G$. The fibers of $H\G$ are 
	$$ H\G_x=\mbox{span }\{X_1(x),\dots,X_{m_1}(x)\},\qquad x\in\G.$$

    A  sub-Riemannian structure is  induced on $\G$ by endowing each fiber of $H\G$ with a scalar product $\scal{\cdot}{\cdot}_{x}$, that makes the basis $X_1(x),\dots,X_{m_1}(x)$ orthonormal,  and the  related norm $|\cdot|_x$. These are defined as follows: if $v=\sum_{i=1}^{m_1} v_iX_i(x)=(v_1,\dots,v_{m_1})$ and $w=\sum_{i=1}^{m_1} w_iX_i(x)=(w_1,\dots,w_{m_1})$ are in $H\G_x$, then $\scal{v}{w}_{x}:=\sum_{j=1}^{m_1} v_jw_j$ and $|v|_x^2:=\scal{v}{v}_{x}$.
	
	The sections of $H\G$ are called {\it horizontal sections}, and for any $x\in\G$, a vector of $H\G_x$ is  a {\it horizontal vector}, while a vector in $T\G_x$ that is not horizontal is called a vertical vector. Each horizontal section is identified by its canonical coordinates with respect to this moving frame  $X_1,\dots,X_{m_1}$. In this way, a horizontal section $\varphi$ is  represented with a function $ \varphi:\G \rightarrow\R^{m_1}$. To simplify the notation, when dealing with two horizontal sections $\phi$ and $\psi$, we drop the index $x$ in the scalar product and norm. 
	
	We collect  some properties of the group operation  and canonical vector fields, see \cite {BLU}.
	
	\begin {prop} The group operation has  the form
	$$  x \circ y = x + y + \mathcal {Q} (x, y), \quad \text{for all } x, y \in \mathbb {G},$$
	where $ \mathcal {Q} = (\mathcal {Q} _1, \ldots, \mathcal {Q} _n): \mathbb {G}\times \mathbb {G} \rightarrow \mathbb {G}$
	and every $ \mathcal {Q} _i $, $ i = 1, \ldots, n $, is a homogeneous polynomial of degree $ \alpha_i $  with respect to the intrinsic dilations of $ \mathbb {G} $, that is,
	$$ \mathcal {Q} _i (\delta_ \lambda (x), \delta_ \lambda (y)) = \lambda ^ {\alpha_i} \mathcal {Q} _i (x, y) \quad \text{for all } x, y \in \mathbb {G}\text{ and }\lambda> 0. $$ Moreover, for all $x, y \in \mathbb {G} $, we have :
	\begin{itemize}
	\item [(i)] $ \mathcal {Q} $ is antisymmetric,  namely,
	$$ \mathcal {Q} _i (x, y) = - \mathcal {Q} _i (-y, -x), \quad \text{for } i = 1, \ldots, n ; $$
	\item [(ii)] $$ \mathcal {Q} _1 (x, y) = \ldots = \mathcal {Q} _ {m_1}  (x,y) = 0,$$
	$$ \mathcal {Q} _i (x, 0) = \mathcal {Q} _i (0, y) = 0 \text{ and } \mathcal {Q} _i (x, x) = \mathcal {Q} _i (x, -x) = 0, \quad \text { for } m_1 <i \leq n,$$
	$$ \mathcal {Q} _i (x, y) = \mathcal {Q} _i (x_1,, x _ {{h_j} -1}, y_1, \ldots, y _ {{h_j} -1}), \quad \text { if } 1 <j \leq k \text { and } i \leq h_j;$$
	\item [(iii)] $$ \mathcal {Q} _i (x, y) = \sum_ {k, h} {\mathcal {R} _ {h, k} ^ i (x, y) (x_ky_h- x_hy_k)},$$ where the functions $ \mathcal {R} _ {h, k} ^ i $ are homogeneous polynomials of degree $ \alpha_i- \alpha_k- \alpha_h $  with respect  to the intrinsic dilations and the sum is extended to all $ h $ and $ k $ such that $ \alpha_h + \alpha_k \leq \alpha_i $.
	\end {itemize} \label {group_operation}
	\end {prop}
	
	The following result is contained in \cite{FSSC03}*{Proposition 2.2}.
	
	\begin {prop} \label {campi-omogenei0} The vector fields $ X_j $ have polynomial coefficients and are of the form
	\begin{equation*}
		X_j (x) = \partial_j + \sum_ {i> h_l} ^ n {q_ {i, j} (x) \partial_i}, \quad \:  j = 1, \ldots, n \: \: \text{ and } \: \: j \leq h_l,
	\end{equation*}
	where $ q_ {i, j} (x) = \frac {\partial \mathcal {Q} _i} {\partial y_j} (x, y) \rvert_ {y = 0} $ and the $ \mathcal {Q} _i$'s are those ones defined in Proposition \ref {group_operation} for $ h_ {l-1} <j \leq h_l $ and $ 1 \leq l \leq k $.  So, if $ h_ {l-1} <j \leq h_l $, then $ q_ {i, j} (x) = q_ {i, j} (x_1, \ldots, x_ {h_ {l-1}}) $ and $ q_ {i, j} (0) = 0 $.
	\end {prop}
	
	\subsection{Intrinsic distance and gauge pseudo-distance} 
	An absolutely continuous curve \linebreak$  \gamma: [0, T] \rightarrow \mathbb {G} $ is called \emph {sub-unitary} with respect to $ X_1, \ldots, X_ {m_1} $ if it is a \emph {horizontal curve}, that is, if there exist real measurable functions $ c_1, \ldots, c_ {m_1}: [0, T] \rightarrow \mathbb {R} $ such that
	$$ \dot {\gamma} (s) = \sum_ {j = 1} ^ {m_1} {c_j (s) X_j (\gamma (s))} \quad  \text {for $ \mathcal {L} ^ 1$-a.e. } s \in [0, T],$$ with $ \sum_ {j = 1} ^ {m_1} {c_j(s) ^ 2} \leq  1 $ for  $ \mathcal {L} ^ 1$-a.e.  $s \in [0, T] $.
	
	\begin {defin} [Carnot-Carath\'eodory distance] Let $ \mathbb {G} $ be a Carnot group. For $ p, q \in \mathbb {G} $, we define their \emph {Carnot-Carath\'eodory distance} $  d_c(p, q) $ as
	$$  d_c (p, q): = \mathrm {inf} \{T> 0: \text{ there exists a sub-unitary curve } \gamma \text { with }\: \gamma (0) = p, \: \gamma (T) = q \}. $$
\end{defin}

We remark that this distance is well defined: the set of sub-unitary curves connecting $ p $ and $ q $ is  nonempty by Chow's theorem  \cite{BLU}*{Theorem 19.1.3}, since by \eqref{stratificazione}, the rank of the Lie algebra generated by $ X_1, \ldots, X_ {m_1} $ is $n$. The Carnot-Carath\'eodory distance $d_c$ induces on $\G$ the same topology as the standard Euclidean  one. We shall denote by $ B_r(p) $ the open  ball, centered in $p$  and of radius $r>0$, associated with $d_c$. For the sake of simplicity, if $p=e$, we will use the notation $B_r(e)=B_r$. 

 The distance $d_c$ is equivalent to a more explicit pseudo-distance, called the \emph{gauge pseudo-distance}, defined as follows.  Let $||\cdot ||$ denote the Euclidean distance to the origin in the Lie algebra $\mathfrak g$. For $u = u_1 + \cdots + u_k \in \mathfrak g$, with $u_i\in V_i$, one defines
\begin{equation}\label{gauge}
	|u|_{\mathfrak g} := \left(\sum_{i=1}^k ||u_i||^{2k!/i}\right)^{\frac{1}{2k!}}.
\end{equation}
The \emph{nonisotropic gauge} in $\G$  is
\begin{equation*}
	|p|_{\G} := |\exp^{-1} p|_{\mathfrak g}, \quad\quad \ \ \ \ \ p\in \G,
\end{equation*}
see \cite{Fo} and \cite{FS}. Since the exponential map $\exp:\mathfrak g \to \G$ is a $C^\infty$-diffeomorphism (actually, analytic), the map $p\to |p|_\G$ is $C^\infty(\G\setminus \{e\})$. Notice that, from \eqref{gauge} and \eqref{dilatazioni}, we have, for any $\lambda>0$ and $p\in\G$,
\begin{equation}\label{gau}
	|\delta_\lambda(p)|_\G = \lambda |p|_\G.
\end{equation}
The \emph{gauge pseudo-distance} in $\G$ is defined by
\begin{equation}\label{rho0}
	d(p,q):=| p^{-1}\circ q|_\G,  \quad\quad \ \ \ \ \ p,q\in \G.
\end{equation}
The function $d$ has all the properties of a distance, except the triangle inequality, which is satisfied with a universal constant, usually different from one, on the right-hand side, see \cites{FS,BLU}. Since the  dilations are group automorphisms, from \eqref{rho0} and \eqref{gau}, it follows that $d$ is homogeneous of degree one with respect to the group dilations, that is, for any $\lambda>0$  and  all $p, p'\in\G$
\begin{equation*}
	d(\delta_\lambda(p),\delta_\lambda(p')) = \lambda d(p,p').
\end{equation*}
It is well known, see, for instance, Proposition 5.1.4 in \cite{BLU}, that there exists a constant $\Lambda_\G$, only depending on the group $\G$, such that, for $p\in \G$,
\begin{equation}\label{rhod}
	\Lambda_\G^{-1} |p|_\G \le d_c(e,p) \le \Lambda_\G |p|_\G.
\end{equation}

The integer
	 \begin{equation}\label{omog-dimen}
	 	Q: = \sum_ {i = 1} ^ k {i\text{dim}(V_i)}
	 \end{equation} 
	is the \emph {homogeneous dimension} of $ \mathbb {G} $. It is the Hausdorff dimension of $ \mathbb {G} \cong \mathbb {R} ^ n $ with respect to the distance $d_c$, see \cite{Mit}.
    
The  $n$-dimensional Lebesgue measure $\mathcal{L}^ n $  is the Haar measure of the group $ \mathbb {G} $, i.e., for every $ \mathcal {L} ^ n $-measurable set $E \subset \mathbb {G}$ and $x\in\G$, it results $\mathcal {L} ^ n \left (x \circ E \right) = \mathcal {L} ^ n \left (E \right)$. Moreover, if $ \lambda> 0 $, then $ \mathcal {L} ^ n \left (\delta_ \lambda \left (E \right) \right) = \lambda ^ Q \mathcal {L} ^ n \left (E \right)$, where $ Q $ is the homogeneous dimension, see \eqref{omog-dimen}. In particular, for any $ r> 0 $  and $ p \in \mathbb {G} $, it holds $$ \mathcal {L} ^ n \left (B_r (p) \right) = r ^ Q \mathcal {L} ^ n \left (B_1 (p)\right) = r ^ Q \mathcal {L} ^ n \left (B_1 \right).$$
 Let us point out that all the spaces $ L ^ p (\mathbb {G}) $  we will deal with are defined with respect to the Lebesgue measure $ \mathcal {L} ^ n $. If $A\subset \G$ is $\mathcal{L}^n$-measurable, we write $|A|=\mathcal{L}^n(A)$. Hereafter, unless differently specified, all considered domains are at least $\mathcal{L}^n$-measurable.

A map $ L: \mathbb {G} \rightarrow \mathbb {R} $ is \emph{$\mathbb {G} $-linear} if it is a group homeomorphism from $ (\G, \circ) $ to $ (\mathbb {R}, +) $ and  is homogeneous of degree 1 with respect to the intrinsic dilations of $ \mathbb {G} $, that is, $L (\delta_ \lambda(x)) = \lambda L (x)$ for $\lambda> 0$ and $x \in \mathbb {G}$. Similarly, we say that a map $\ell:\G\to\R$ is $\G$-affine if there exists a $\mathbb {G} $-linear map $L$ and $c\in\R$ such that $\ell(x)=L(x)+ c$ for every $x\in\G$. Given a basis $(X_1,\ldots, X_{n})$, all $\G$-linear maps are represented as follows.

\begin{prop}\label{rap-lin}
	A map $L:\G\to\R$ is $\G$-linear if and only if there is $a=(a_1,\ldots,a_{m_1})\in\R^{m_1}$ such that, if $x=(x_1,\ldots,x_n)\in\G$, then $$L(x)=\sum_{i=1}^{m_1}a_ix_i.$$
\end{prop}

Moreover, if $x=(x_1,\ldots,x_n)\in\G$  and $x_0\in\G$ are given, we set 
\begin{equation}\label{map-pi}
	\pi_{x_0}(x):=\sum_{j=1}^{m_1}x_jX_j(x_0).
\end{equation}

\subsection{Folland-Stein and horizontal Sobolev classes}
Fixed $\Omega \subseteq \G$, the action of vector fields $X_j$, with $j=1,\ldots,m_1$, on a function $f:\Omega\to\R$ is specified by the Lie derivative. We say that a  function $f$ is differentiable along  the direction $ X_j $  at $ x_0 \in \G $, when $ j \in \{1, \ldots, m_1\} $, if the map $ \lambda \mapsto f (\tau_ {x_0} (\delta_ \lambda (e_j))) $ is differentiable  at $ \lambda = 0 $, where $ e_j $ is the $ j $-th vector of the canonical basis of $ \mathbb {R} ^ {m_1} $. In this case, we will write
$$ X_j f (x_0) = \left. \frac {d} {d \lambda} f (\tau_ {x_0} (\delta_ \lambda  (e_j))) \right | _ {\lambda = 0}.$$ 
If, instead, $ f \in L ^ 1_ {\mathrm{loc}} (\Omega)$, $ X_jf $ exists in a weak sense if 
$$ \int_ \Omega {fX_j \varphi} \: d\mathcal{L}^n = - \int_ \Omega {\varphi X_jf} \: d\mathcal{L}^n$$
for each $ \varphi \in C _0 ^ \infty (\Omega) $.

Let us fix a basis $X_1,\ldots, X_{m_1}$ of the horizontal layer. For any function $f:\Omega\to \R$ for which the partial derivatives $X_jf$, $j=1,\ldots,m_1$, exist, we define the horizontal gradient of $f$, denoted by $\Hnabla f$, as the horizontal section
\begin{equation}\label{horiz-grad}
	\Hnabla f:=\sum_{i=1}^{m_1}(X_if)X_i,
\end{equation}
whose coordinates are $(X_1f,\ldots,X_{m_1}f)$. Moreover, if $\phi=(\phi_1,\dots,\phi_{m_1})$ is a horizontal section such that $X_j\phi_j\in L^1_{ \mathrm{loc}}(\G)$ for $j=1,\dots,m_1$, we define $\divg\phi$ as the real-valued function
\begin{equation*} 
	\divg\phi:=-\sum_{j=1}^{m_1}X_j^*\phi_j=\sum_{j=1}^{m_1}X_j\phi_j.
\end{equation*}

The positive \emph{sub-Laplacian} operator on $\G$ is the second-order differential  operator given by
$$
\mathcal{L}:= \sum_{j=1}^{m_1}X^*_jX_j=-\sum_{j=1}^{m_1} X_j^2.
$$
It is easy to see that
\begin{equation}\label{subLapl-diver}
	\mathcal{L}u = - \divg (\Hnabla u).
\end{equation}
The operator $\mathcal L$ is left-invariant with respect to group translations and homogeneous of degree two with respect to group  dilations, i.e., for any $x\in\G$ and $\lambda>0$, we have
$$\mathcal L (u\circ \tau_x) = (\mathcal Lu)\circ \tau_x, \qquad \mathcal{L}(\delta_\lambda (u)) = \lambda^2 \delta_\lambda(\mathcal{L} u).$$
Furthermore, by the  assumption \eqref{stratificazione}, the system $\{X_1,\ldots,X_{m_1}\}$ satisfies the finite rank condition 
\[
\text {rank Lie} [X_1,\ldots,X_{m_1}] = n,
\] 
therefore by  H\"ormander's theorem \cite{H} the operator $\mathcal{L}$ is hypoelliptic. However, when the step $k$ of $\G$ is greater than one, the operator $\mathcal{L}$ fails to be elliptic at every point of $\G$.

The H\"older classes $C^{k, \alpha}$ had been introduced by Folland and Stein, see \cites{Fo, FS}. The functions in these classes are H\"older continuous with respect to the metric $d_c$.

\begin{defin}[Folland-Stein classes]\label{eq:fsclass}
	Let $\Omega \subseteq \G$  be an open set, $0<\alpha\leq 1$ and $u:\Omega\to \R$. We say that
	$u\in C^{0, \alpha}(\Omega)$ if there exists a constant $M>0$ such that 
	$$|u(x) - u(y)| \leq M\,  d_c(x,y)^{\alpha} \qquad\text{for every } x,y\in\Omega.$$Defining the H\"older seminorm of $u\in C^{0,\alpha}(\Omega)$ as 
	\begin{equation*} 
		[u]_{C^{0,\alpha}(\Omega)}:= \underset{\underset{x \neq y}{x,y\in \Omega}}{\sup} \frac{|u(x)-u(y)|}{ d_c(x,y)^{\alpha}},
	\end{equation*}
	the space $C^{0,\alpha}(\Omega)$ is a Banach space  with the norm $$\|u\|_{C^{0,\alpha}
		(\Omega)}:= \|u\|_{L^\infty(\Omega)} + [u]_{C^{0,\alpha}(\Omega).}$$ 
	For any $k\in \mathbb N$, we say that $u \in C^{k, \alpha}(\Omega)$ if $X_i u \in C^{k-1,\alpha}(\Omega)$ for every $i=1,\ldots, m_1$.
\end{defin}

Note  that $C^{0,1}(\Omega)$ coincides with the class of Lipschitz continuous functions with respect to the metric $ d_c$ on $\Omega$. Also, we denote by $C^{k,\lambda}(\Omega,H\G)$ the space of all the horizontal sections $\phi:\Omega\to H\G$, $\phi:=(\phi_1,\ldots,\phi_{m_1})$, such that $\phi_j\in C^{k,\lambda}(\Omega)$ for $j=1,\ldots,m_1$.
\begin{defin}[Horizontal Sobolev spaces]\label{def:hsob}
	Given an open set $\Omega \subseteq \G$,  $u:\Omega\to \R$ and $1\leq p<\infty$, the horizontal Sobolev spaces are defined as 
	$$ HW^{1,p}(\Omega):=\left\{u\in L^p(\Omega): \ X_ju\in L^p(\Omega)\ \text{for all }\, j=1,\ldots, m_1\right\},$$ 
	which is a Banach space with the norm $$\|u\|_{HW^{1,p}(\Omega)} := \|u\|_{L^p(\Omega)} + \|\Hnabla u\|_{L^p(\Omega, H\G)}.$$
\end{defin}

The subspace $HW^{1,p}_0(\Omega)$ of $HW^{1,p}(\Omega)$  is defined as the closure of $C^\infty_0(\Omega)$ with respect to the norm $\|\cdot\|_{HW^{1,p}(\Omega)}$.

\section{Dichotomy results}\label{sec:dichotomy}

	The first step for reaching Theorem \ref{theor-Lipsch-contin-alm-minim} is to prove a \emph{dichotomy} result. Roughly speaking, two situations can occur: either the average of the energy of an almost minimizer decreases in a smaller ball, or the distance  between its horizontal gradient and a suitable constant horizontal section becomes as small as we wish (that is, $\G$-linear functions are the ''only ones for which the average does not improve in small balls''). Before passing to the proof of this result, let us state the following remark.
    \begin{remark}\label{rmk-F}
	Let be $\ell:\G\to\R$ a $\G$-affine map. By Proposition \ref{rap-lin}, there exist a suitable $q=(q_1,\ldots,q_{m_1})\in\R^{m_1}$ and $c\in\R$  such that $$\ell(x)=\sum_{j=1}^{m_1}q_ix_i+c \qquad \text{for any }x\in\G.$$ Recalling  \eqref{horiz-grad}, by Proposition \ref{campi-omogenei0}, we  directly get 
	\begin{equation}\label{grad-aff}
		\Hnabla \ell(x)=\sum_{j=1}^{m_1} q_jX_j(x)=:\bq(x), \qquad \text{for any }x\in\G.  
	\end{equation} 
	Thus, the horizontal gradient of a $\G$-affine map $\ell$ is a horizontal section $\bq$, which has constant components with respect to the moving frame given by generating vector fields. From now on, we call such a horizontal section a \emph{constant} horizontal section. Note, however, that a constant horizontal section still depends on the point $x\in\G$ due to the moving frame. We  also observe  that all $\G$-affine functions $\ell$ as before are $\G$-harmonic in $\G$, in the sense that $\mathcal{L}\ell=0$ in $\G$. Indeed, by  \eqref{subLapl-diver} and \eqref{grad-aff}, it immediately holds $\mathcal{L}\ell=0$ in $\G$. We refer to \cite{BLU} for further characterizations of $\G$-harmonicity within Carnot groups.
\end{remark}

Our first result is the following.
\begin{prop}\label{dic} Let $u \in HW^{1,2}(B_1)$ be such that
	\begin{equation}\label{first}J(u,B_1) \leq (1+\sigma) J(v,B_1)\end{equation}
	for all $v \in HW^{1,2}(B_1)$ such that $u-v\in HW^{1,2}_0(B_1)$.
	Denoting by
	\begin{equation}\label{def_a}
    a : = \left(\fint_{B_1} | \Hnabla u(x)|^2 dx\right)^{1/2},
        \end{equation}
	there exists $\e_0\in(0,1)$ such that, for every $\e\in(0,\e_0)$ there exist $\eta\in(0,1)$,  $M\ge1$, and $\sigma_0\in(0,1)$, depending on $\e$ and $Q$,
	such that, if $\sigma\in [0, \sigma_0]$ and $a\ge M$, then the following dichotomy holds. Either
	\begin{equation}\label{alt1}\left(\fint_{B_{\eta}} | \Hnabla u(x)|^2 dx\right)^{1/2} \leq \frac a 2, \end{equation}
	or
	\begin{equation}\label{alt2}\left(\fint_{B_\eta} | \Hnabla u(x)- \mathbf{q}(x)|^2 dx\right)^{1/2} \leq \e a,\end{equation}
	where $\mathbf{q}:\G \to H\G$ is a constant horizontal section, see \eqref{grad-aff}, with 
	\begin{equation}\label{bound-q}
		\frac a 4 < |\mathbf{q}| \leq C_0 a,
	\end{equation}
	for some universal constant $C_0>0$.
\end{prop}

\begin{proof} We split the proof into several steps.\medskip
	
	\noindent{\em Step 1: pointwise estimates. }
	Let $v:B_1 \to \R$ be the $\G$-harmonic replacement of $u$ in $B_1$, that is, the unique solution of Dirichlet problem 
	\begin{equation*}
			\begin{cases}
				\mathcal{L}v=0\quad \textrm{in\ }B_1,\\ 
			u-v\in HW^{1,2}_0(B_1),
			\end{cases}
	\end{equation*}
	or, equivalently, the minimizer of Dirichlet energy among all competitors $w \in HW^{1,2}(B_1)$ such that $u-w\in HW^{1,2}_0(B_1)$, i.e.,
	\begin{equation}\label{g-minim}
		\int_{B_1}\left|\Hnabla v(x)\right|^2 \,dx=\min_{u-w\in HW^{1,2}_0(B_1)}\int_{B_1}\left|\Hnabla w(x)\right|^2 \,dx.
	\end{equation}
	
	By Theorem 1.1 in \cite{CM22},  there exists a constant $C_0>0$,  only depending  on $Q$, such that
	\begin{equation*}
		\sup_{B_{1/2}}\left|\Hnabla v\right|\leq C_0 \left(\fint_{B_1} |\Hnabla v(x)|^2 dx\right)^{1/2}.
	\end{equation*}  
	Then, recalling \eqref{def_a}, by  the minimality \eqref{g-minim} of $v$ for the Dirichlet energy, we conclude
	\begin{equation}\label{stima-grad-v}
		\left| \Hnabla v\right|\leq C_0 a \quad  \text{in }B_{1/2}.
	\end{equation}  
	\medskip
	
	\noindent{\em Step 2: oscillation estimates.}
	By Theorem 1.3  in \cite{CM22}, we have the following oscillation estimates: for all $\eta\in(0,1/2]$, it holds
	\begin{equation*}
		\max_{1\leq i\leq m_1} \mathrm{osc}_{B_\eta}X_i v \leq  C \bigg(\frac{\eta}{1/2}\bigg)^{\alpha} \left(\fint_{B_1} |\Hnabla v(y)|^2 dy\right)^{1/2},
	\end{equation*} 
	for some $\alpha=\alpha(Q)\in (0,1]$ and universal constant $C>0$.
	Therefore, we have 
	\begin{equation}
		\max_{1\leq i\leq m_1} |X_i v(x)-X_iv(0)|\leq \max_{1\leq i\leq m_1} \mathrm{osc}_{B_\eta}X_i v \leq  C \bigg(\frac{\eta}{1/2}\bigg)^{\alpha}\left(\fint_{B_1} |\Hnabla v(y)|^2 dy\right)^{1/2} \quad \text{for }x \in B_\eta. 
	\end{equation}
	 Thus, for any $i=1,\dots, m_1$, denoting by $q_i=X_iv(0)$  and $\mathbf{q}:\G\to HG$ the constant horizontal section defined as in \eqref{grad-aff}, it  results
	\begin{equation}
		| \Hnabla v-\mathbf{q}|\leq  C_1\bigg(\frac{\eta}{1/2}\bigg)^{\alpha} \left(\fint_{B_1} |\Hnabla v(y)|^2 dy\right)^{1/2}\quad \text{in }B_\eta,
	\end{equation}
	for some universal constant $C_1>0$. This gives that, for all $\eta\in (0,1/2]$,
	\begin{equation}\label{average-integral-nabla-v-q}\begin{split}
			\fint_{B_\eta}\left| \Hnabla v(x)-\mathbf{q}(x)\right|^2\,dx\le\;& \fint_{B_\eta}\bigg(C^2_1\bigg(\frac{\eta}{1/2}\bigg)^{2\alpha} \fint_{B_1} | \Hnabla v(y) |^2 \,dy\,\bigg)\,dx\\
			=\;&  C_2\,\eta^{2\alpha} \fint_{B_1} |\Hnabla v(y) |^2 d y\\
			\le\;&  C_2\,\eta^{2\alpha}\,a^2,
	\end{split}\end{equation}
	for some $\alpha\in(0,1)$ and $C_2>0$ universal.
	\medskip
	
	\noindent{\em Step 3: proximity to the $\G$-harmonic replacement.}
	 We know that
	\begin{align}\label{pr-scalare-1}
		\int_{B_1} |\Hnabla u -\Hnabla v|^2 \,dx &\leq J(u, B_1) + \int_{B_1}(|\Hnabla v|^2 - 2 \left \langle \Hnabla u,\Hnabla v \right \rangle)\,dx \\ \nonumber
		&=J(u, B_1) - \int_{B_1}(|\Hnabla v|^2 + 2 \left \langle \Hnabla (u-v),\Hnabla v \right \rangle)\,dx.       
	\end{align} 
	On the other hand, since $X^*_j=-X_j$ for $j=1,\ldots,m_1$ and $u-v\in HW_0^{1,2}(B_1)$, it occurs
	\begin{align}\label{weak-harm}
		\int_{B_1} \left\langle\Hnabla(u-v),\Hnabla v\right\rangle \,dx&=\sum_{j=1}^{m_1}\,\int_{B_1}X_j(u-v)X_jv\,dx\\
		&=-\sum_{j=1}^{m_1}\,\int_{B_1}(u-v)X^2_jv\,dx=\int_{B_1}(u-v)\mathcal{L}v \,dx=0,  \nonumber  \end{align}        
	 which together with \eqref{pr-scalare-1} gives
	\begin{align}\label{pr-scalare-2}
		\int_{B_1} |\Hnabla u(x) -\Hnabla v(x)|^2 \,dx \leq J(u, B_1) - \int_{B_1}|\Hnabla v(x)|^2 dx.       
	\end{align}
	Then, using the hypothesis \eqref{first} and the minimality of $v$ in \eqref{g-minim},  for $\sigma>0$ to be made precise later, we obtain
	\begin{equation*}\begin{split}
			\int_{B_1}\big(\left| \Hnabla u(x)-\Hnabla v(x)\right|^2\big)\,dx	
			&\le  (1+\sigma)J(v,B_1)-\int_{B_1}\left| \Hnabla v(x)\right|^2\,dx \\
			&\le C\left(\sigma\int_{B_1}\left| \Hnabla v(x)\right|^2\,dx+1\right)\\
			&\le C\left(\sigma\int_{B_1}\left| \Hnabla u(x)\right|^2\,dx+1\right)
	\end{split}\end{equation*}
	for some suitable positive constant  $C$ only depending on $Q$. Consequently, taking the average over $B_1$, we  achieve
	\begin{equation*}
		\fint_{B_1}\left| \Hnabla u(x)-\Hnabla v(x)\right|^2\,dx\le C(\sigma a^2+1).
	\end{equation*}
	Using this and \eqref{average-integral-nabla-v-q}, we  get
	\begin{equation}\label{average-integral-nabla-u-q-p->-2}\begin{split}
			\fint_{B_\eta}\left| \Hnabla u(x)-\bq(x)\right|^2\,dx &\le2
			\fint_{B_\eta}\big(\left| \Hnabla u(x)-\Hnabla v(x)\right|^2+\left| \Hnabla v(x)-\bq(x)\right|^2\big)\,dx\\
			&\le 2\left(\frac{\left|B_1\right|}{\left|B_\eta\right|}\fint_{B_1}\left| \Hnabla u(x)-\Hnabla v(x)\right|^2\,dx+ C_2\,\eta^{2\alpha }\,a^2\right)\\
                &\le 2\Big(C\eta^{-Q}(\sigma a^2+1)+C_2a^2\eta^{2\alpha }\Big),\\			&\le 2C \eta^{-Q}\sigma a^2+2C\eta^{-Q}+2C_2a^2\eta^{2\alpha },
		\end{split}
	\end{equation}
	and hence
	\begin{equation}\label{average-integral-nabla-u-to-p-B-eta-p->-2}
		\fint_{B_\eta}\left| \Hnabla u(x)\right|^2\,dx\le 4C\eta^{-Q}\sigma a^2+4C\eta^{-Q}+4 C_2a^2\eta^{2\alpha}+2\left|\bq\right|^2.
	\end{equation}
	\medskip

	\noindent{\em Step 4: perturbative estimates.}		
	Now, given $\e_0\in(0, 1/4]$, we claim that 
	for every $\e\in (0,\e_0)$, there exists $\eta$ small enough (depending on $\e$)
	such that if $\sigma$ is sufficiently small and $a$ sufficiently large (depending on $\eta$,
	and thus on $\e$), then
	\begin{equation}\label{conditions-on-eta-sigma-a}
		4C\eta^{-Q}\sigma a^2+4C\eta^{-Q}+4 C_2a^2\eta^{2\alpha}\leq  2\e^2a^2\le \frac{a^2}{8}.
	\end{equation}
	 Let $\eta>0$ sufficiently small such that $\e^2-2C\eta> 2 C_2\eta^{2\alpha}$. So, defining
	$$
	M:=\left(\frac{2C\eta^{-Q}}{\e^2-2C\eta-2 C_2\eta^{2\alpha}}\right)^{1/2},
	$$
	we can suppose $M\ge1$, by taking $\eta$ small enough.
	 Therefore, for every $a\geq M$  and $0\leq\sigma\le\eta^{Q+1}\allowbreak =:\sigma_0$, we have, exploiting the expression of $M$,
	\begin{eqnarray*}
		&&4C\eta^{-Q}\sigma a^2+ 4C\eta^{-Q}+ 4C_2a^2\eta^{2\alpha}\\
		&\leq & a^2\Big(4C\eta+ 4 C_2\eta^{2\alpha}\Big) + 4C\eta^{-Q}\\
		&=& a^2\Big(4C\eta+ 4 C_2\eta^{2\alpha}\Big)+ 2\,M^2\Big(\e^2-2C\eta-2 C_2\eta^{2\alpha }\Big)\\
		&\le& a^2\Big(4C\eta+ 4 C_2\eta^{2\alpha}\Big)+ 2\,a^2\Big(\e^2-2C\eta-2 C_2\eta^{2\alpha }\Big)\\
		&=& 2\e^2 a^2,
	\end{eqnarray*}
	which proves \eqref{conditions-on-eta-sigma-a}, choosing $\e$ small enough. Note that, by Step 2, $\eta\leq 1/2 $, thus $\sigma_0<1$.
	\medskip
	
	\noindent{\em Step 5: conclusion of the proof.}		
	In order to complete the proof of Proposition \ref{dic}, we distinguish two cases.  First, we suppose that
	\[\left|\mathbf{q}\right|\le\frac{a}{4}.\]
	Then, by \eqref{average-integral-nabla-u-to-p-B-eta-p->-2}   and \eqref{conditions-on-eta-sigma-a}, we conclude that
	\begin{align*}
		\fint_{B_\eta}\left| \Hnabla u(x)\right|^2\,dx\le \frac{a^2}{8}+ 2\frac{a^2}{16}=\frac{a^2}{4},
	\end{align*}
	 from which we immediately deduce the first alternative  \eqref{alt1}.
	
	Otherwise,  we assume, recalling \eqref{stima-grad-v},
	\[ \frac{a}{4}<\left|\bq\right|\le C_0\,a,\]
	and therefore, by \eqref{average-integral-nabla-u-q-p->-2} and \eqref{conditions-on-eta-sigma-a},  the second alternative  \eqref{alt2} directly follows.
\end{proof}	

Now, we show that the alternative \eqref{alt2} in Proposition \ref{dic} can be "improved" when $\e$ and $\sigma$ are sufficiently small. This result is the counterpart of Lemma 2.3 in \cite{DS} in the more general setting of Carnot groups of step two.

\begin{lem}\label{lemma-second-alternative-dichotomy-improved}
	Let $u\in HW^{1,2}(B_1)$ be such that $u\ge0$ a.e. in $B_1$ and
	\begin{equation}\label{hp-bis}
		J(u,B_1)\le (1+\sigma) J(v,B_1)
	\end{equation}
	for all $v\in HW^{1,2}(B_1)$ such that $u-v\in HW^{1,2}_0(B_1)$.
	
	Let
	\begin{equation*} 
		a:=\left(\fint_{B_1}\left| \Hnabla u(x)\right|^2 \,dx\right)^{1/2}
	\end{equation*}
	and suppose that 
	\begin{equation}\label{316BIS}
		a_1\ge a\ge a_0>0,\end{equation}
	for some suitable $a_1\ge a_0$.
	Assume  also
	\begin{equation}\label{second-alternative-dichotomy}
		\left(\fint_{B_1}\left| \Hnabla u(x)-\mathbf{q}(x)\right|^2 \,dx\right)^{1/2}\leq\e a,
	\end{equation}
	for some constant horizontal section $\mathbf{q}:\G\to H\G$  such that
	\begin{equation}\label{estimate-norm-q}
		\frac{a}{8}<|\mathbf{q}|\le 2C_0a,
	\end{equation}
	where $C_0>0$ is the universal constant in Proposition \ref{dic}.

	There exists $\alpha_0\in(0,1]$, only depending on $Q$, such that, given $0<\alpha<\alpha_0$, there exist $\rho=\rho(\alpha,Q)>0$, $\e_0= \e_0(Q,\alpha, a_0,a_1) >0$, and $c_0=c_0(Q, \alpha, a_0, a_1) >0,$ such that if $$ \e \leq \e_0 \quad \mbox{and} \quad  \sigma \leq c_0 \e^2,$$then  
	\begin{equation}\label{tesi-lemma-improv}
		\left(\fint_{B_\rho}\left| \Hnabla u(x)-\widetilde{\mathbf{q}}(x)\right|^2\,dx\right)^{1/2}\le\rho^{\alpha}\e a,
	\end{equation}
	where $\mathbf{\widetilde{q}}:\G \to H\G$ is a constant horizontal section, see \eqref{grad-aff}, for some suitable $\widetilde{q}=(\widetilde{q}_1,\ldots,\widetilde{q}_{m_1})\in \R^{m_1}$ such that 
	\begin{equation}\label{so3cer5b56859tj45ivnt45-2}
		\left|\mathbf{q}-\mathbf{\widetilde{q}}\right|\le \widetilde{C}\e a.
	\end{equation}
	for some universal constant $\widetilde{C}>0$.
\end{lem}

\begin{proof}[Proof of Lemma \ref{lemma-second-alternative-dichotomy-improved}]
	We divide the proof into several steps.\medskip
	
	\noindent{\em Step 1: energy estimates for the $\G$-harmonic replacement and comparison of energies.}		
	Let us  set $\tau:=\frac{1}{10\Lambda_\G^2}<1$, where $\Lambda_\G\geq 1$  is the structural constant given by \eqref{rhod}  only depending on $\G$.  Let $\bar{v}$ denote the $\G$-harmonic replacement of $u$ in $B_{\tau}$  and $v$ be defined as
	\begin{equation}\label{defin-of-special-competitor}
		v:=\begin{cases}
			\bar{v}&\mbox{in }B_{\tau},\\
			u&\mbox{in }B_1\setminus B_{\tau}.
		\end{cases}
	\end{equation}
	 We first note that $v\in HW^{1,2}(B_1)$ and $u-v\in HW^{1,2}_0(B_1)$, thus, by hypothesis \eqref{hp-bis}, we  have \[J(u,B_1)\le (1+\sigma)J(v,B_1),\]  which leads to
	\begin{align*}
		J(u,B_{\tau}) + J(u,B_1\setminus B_{\tau})=\;&J(u,B_{1})\nonumber \\
		\le\;& (1+\sigma)\,J(v,B_1)\nonumber  \\ 
        =\;& J(v,B_1)+\sigma J(v,B_1)\\
		=\;& J(v,B_{\tau})+J(v,B_1\setminus B_{\tau})+\sigma J(v,B_1)\\
		=\;&J(v,B_{\tau})+J(u,B_1\setminus B_{\tau})+\sigma J(v,B_1),\nonumber	\\			
	\end{align*}
	 namely,
    	$$J(u,B_{\tau})\le J(v,B_{\tau})+\sigma  J(v,B_1).$$
	 By definition of $J$ in \eqref{def-J},  this reads as
	\begin{align*}
		\int_{B_{\tau}}\left|\Hnabla u(x)\right|^2\,dx+\left|\left\{u>0\right\}\cap B_{\tau}\right|\leq &\int_{B_{\tau}}\left|\Hnabla v(x)\right|^2\,dx+\left|\left\{u>0\right\}\cap B_{\tau}\right|+\sigma J(v,B_1)\\
		\leq & \int_{B_{\tau}}\left|\Hnabla v(x)\right|^2\,dx+\left|B_{\tau}\right|+\sigma J(v,B_1),
	\end{align*}
	which, since $u\geq 0$ a.e. in $B_1$ by assumption,  yields
	\begin{equation}\label{diff-norm-grad-zero-lev-set}
		\int_{B_{\tau}}\big(\left| \Hnabla u(x)\right|^2-\left| \Hnabla v(x)\right|^2\big)\,dx
		\le \left|\left\{u=0\right\}\cap B_{\tau}\right|+\sigma J(v,B_1).
	\end{equation}
	Moreover, by definition of $v$ in \eqref{defin-of-special-competitor}, we have
	\begin{equation}\begin{split}\label{J-v-B1}
			J(v,B_1) =\;& \int_{B_1}\Big(|\Hnabla v(x)|^2+\chi_{\{v>0\}}(x)\Big)\,dx\\
			\le\;&\int_{B_{\tau}}|\Hnabla v(x)|^2\,dx +\int_{B_1\setminus B_{\tau}}|\Hnabla v(x)|^2\,dx+|B_1|\\ 
			=\; & \int_{B_{\tau}}|\Hnabla \bar{v}(x)|^2\,dx +\int_{B_1\setminus B_{\tau}}|\Hnabla u(x)|^2\,dx+|B_1|\\ 
			\le\;&\int_{B_1}\left|\Hnabla u\right|^2\,dx+\left|B_1\right|\le\;\left|B_1\right| (a^2+1),
	\end{split}\end{equation}	
	where in the second inequality we use the fact that $\bar{v}$ is the $\G$-harmonic replacement of $u$ in $B_{\tau}$ and therefore $\bar{v}$  minimizes the Dirichlet energy on $B_\tau$.
	
	Finally, recalling \eqref{weak-harm}, by \eqref{diff-norm-grad-zero-lev-set} and \eqref{J-v-B1} we obtain that
	\begin{align}\label{second-ineq-condition-almost minim}
		\int_{B_{\tau}}\left|\Hnabla u(x)-\Hnabla v(x)\right|^2\,dx  =\; & \int_{B_{\tau}}\left(\left|\Hnabla u(x)\right|^2-2 \left \langle \Hnabla u, \Hnabla v \right \rangle + \left|\Hnabla v(x)\right|^2\right)\,dx \\ \nonumber =\; & \int_{B_{\tau}}\left(\left|\Hnabla u(x)\right|^2-2 \left \langle \Hnabla (u-v), \Hnabla v \right \rangle-\left|\Hnabla v(x)\right|^2\right)\,dx\\ \nonumber=\; & 
		\int_{B_{\tau}}\left(\left|\Hnabla u(x)\right|^2-\left|\Hnabla v(x)\right|^2\right)\,dx\\ \nonumber \leq \; &                  \left|\left\{u=0\right\}\cap B_{\tau}\right|+\sigma J(v,B_1)\\
		\leq\; &  \left|\left\{u=0\right\}\cap B_{\tau}\right|+|B_1|\sigma(a^2+1).\nonumber
	\end{align}
	\medskip
	
	\noindent{\em Step 2: measure estimates for the zero level set.}		
	Now, we claim that
	\begin{equation}\label{estimate-zero-set-almost minim}
		\left|B_{\tau}\cap \left\{u=0\right\}\right|\le C_1\e^{2+\delta},
	\end{equation}
	for some $C_1>0$ and $\delta>0$.  \medskip
	
	\noindent{\em Step 2.1: comparison with a $\G$-affine function.}		
	To prove \eqref{estimate-zero-set-almost minim}, we consider  a function $\ell :\G \to \R$ defined by 
	\begin{equation}\label{defin-linear-funct}
		\ell(x):= b+\left \langle \mathbf{q}(x), \pi_x(x) \right \rangle=b+\sum_{j=1}^{m_1}q_jx_j,\qquad  b:= \fint_{B_1}u(x)\,dx.
	\end{equation}
	We remark that
	\begin{align*}
		 (u-\ell)_{B_1}:=\fint_{B_1}\big(u(x)-\ell(x)\big)\,dx=b-\left(b+\fint_{B_1}\sum_{j=1}^{m_1} q_jx_j\,dx\right)=0,
	\end{align*}
	where  the last equality is a consequence of the symmetry with respect to the identity element of the Carnot-Carath\'eodory ball $B_1$. Then, by the Poincar\'e inequality (see, e.g., \cite{Jer}) we have
	\begin{equation*} 
		\left\|u-\ell-(u-\ell)_{B_1}\right\|_{L^2(B_1)}=
		\left\|u-\ell\right\|_{L^2(B_1)}\le C\left\|\Hnabla (u-\ell)\right\|_{L^2(B_1)},
	\end{equation*}
	for some $C>0$ universal.
	
	Since by Proposition \ref{campi-omogenei0}, $\Hnabla \ell =\mathbf{q}$,  this, together with hypothesis \eqref{second-alternative-dichotomy}, leads to
	\begin{equation}\label{stima-lp-u-l}
		 \fint_{B_1}\left|u(x)-\ell(x)\right|^2\,dx\le C\e^2a^2.
	\end{equation}
	Finally, we  point out that, since by assumption $u\ge 0$, it  holds $\ell^-\le \left|u-\ell\right|$, so by \eqref{stima-lp-u-l}, we  obtain
	\begin{equation}\label{ineq-average-negative-part}
		\fint_{B_1}(\ell^-(x))^2\,dx\le C\e^2a^2,
	\end{equation}
	for some $C>0$ universal.
	\medskip
	
	\noindent{\em Step 2.2: lower bounds on the $\G$-affine function.}		
	 We want to show that if $\e$ is sufficiently small,  then
	\begin{equation}\label{lower-bound-linear-funct-1}
		\ell\ge c_1a\quad\mbox{in }B_{\tau},
	\end{equation}	
	for some $c_1>0$. To prove this, we argue by contradiction assuming that
	$$ \min_{x\in \overline{B_{\tau}}}\ell(x)<ca$$
	for any $c>0$. 
	We notice that, for every $x\in B_{\tau}$,  recalling $\tau=\frac{1}{10\Lambda_\G^2}$  and \eqref{rhod}, we get
	\begin{equation*}
		\left | \ell(x)-b \right |=\left | \left \langle \mathbf{q}(x),\pi_x(x) \right \rangle \right |\leq \left | \mathbf{q} \right |\left | x^{(1)} \right |\leq \left | \mathbf{q} \right |\left | x \right|_{\G}\leq \left | \mathbf{q} \right | \Lambda_\G \,d_c(e,x)\leq \frac{\left | \mathbf{q} \right |}{10\Lambda_\G},
	\end{equation*}
	where $|x^{(1)}|$ denotes the Euclidean norm on $\R^{m_1}$. Consequently,
    
$$ -\frac{\left | \mathbf{q} \right |}{10\Lambda_\G}\le \ell(x)-b\le\frac{\left | \mathbf{q} \right |}{10\Lambda_\G} \qquad \text{ for any }x\in B_{\tau},$$
	and therefore
	$$  ca>\min_{x\in \overline{B_{\tau}}}\ell(x)\ge b-\frac{\left | \mathbf{q} \right |}{10\Lambda_\G},$$
	 which gives
	\begin{equation}\label{stima-sopra-b}
		b\le  ca +\frac{|\mathbf{q}|}{10\Lambda_\G},
	\end{equation}
	for any $c>0$.
	Now, taking into account the usual identifications given by exponential coordinates, let us consider 
	\begin{align*}
		{\mathcal{B}}:=\Big\{x=\left [ x^{(1)},x^{(2)} \right ]\in\, &\R^{m_1}\times \R^{m_2}\equiv \G:\;x^{(1)}_j=-\frac{tq_j}{|\mathbf{q}|}+\eta_j,\; \text{for }j=1,\ldots,m_1; \\
		&x^{(2)}_i=\xi_i,\; \text{for }i=m_1+1,\ldots,m_2, \;{\mbox{ 
				for some }} (t,\eta,\mathbf{\xi})\in\mathcal{A}\Big\},
	\end{align*}
	where we set 
	\begin{align*}
		\mathcal{A}:=\Bigg\{ (t,\eta,\mathbf{\xi})\in \R\times \R^{m_1}\times \R^{m_2}:\; &t\in\left[\frac{1}{4\Lambda_\G},\frac{3}{10\Lambda_\G}\right],\\
        &\sum_{j=1}^{m_1}\eta_j^2\leq \frac{1}{100\Lambda_\G^2}\;\;{\mbox{and}}\sum_{i=m_1+1}^{m_2}\xi_i^2\leq \frac{1}{100\Lambda_\G^4} \Bigg\}.
	\end{align*}
	We observe that if $x\in{\mathcal{B}}$, then 
	\begin{equation*}\begin{split}                d_c(e,x)^4&\leq\Lambda_\G^4|x|_{\G}^4 = \Lambda_\G^4\Big(\sum_{j=1}^{m_1}\big(-\frac{tq_j}{|\bq|}+\eta_j\big)^2\Big)^2+\Lambda_\G^4\sum_{i=m_1+1}^{m_2}\xi_i^2\\ &\leq 4\Lambda_\G^4\Big(t^2+\sum_{j=1}^{m_1}\eta_j^2\Big)^2 +\Lambda_\G^4\sum_{i=m_1+1}^{m_2}\xi_i^2 \leq 4\Lambda_\G^4 \Big(\frac{9}{100\Lambda_\G^2}+\frac1{100\Lambda_\G^2}\Big)^2+ \frac{1}{100}<1,
		\end{split}
	\end{equation*}
	so we have
	\begin{equation}\label{calB-in-B1}
		{\mathcal{B}}\subseteq B_1.
	\end{equation}
	Furthermore, by \eqref{stima-sopra-b}, we  have, for $x\in{\mathcal{B}}$,
	\begin{eqnarray*}
		\ell(x)=b-t|\mathbf{q}|+\sum_{j=1}^{m_1}q_j\eta_j \le ca +\frac{|\mathbf{q}|}{10\Lambda_\G}-\frac{|\mathbf{q}|}{4\Lambda_\G}+\frac{|\mathbf{q}|}{10\Lambda_\G}
		= ca- \frac{ |\mathbf{q}|}{20\Lambda_\G}.
	\end{eqnarray*}
	 which implies, using hypothesis \eqref{estimate-norm-q}, 
	$$ \ell(x)\leq ca - \frac{ a}{160\Lambda_\G}.$$
	Then, taking $c\in\left(0, \frac{1}{320\Lambda_\G}\right)$, we  infer that $$ \ell(x)\le -\frac{a}{320\Lambda_\G}.$$
	
	Accordingly,  exploiting \eqref{calB-in-B1} into \eqref{ineq-average-negative-part}, we  obtain
	\begin{eqnarray*}
		&& C\,|B_1|\,\varepsilon^2a^2\ge
		\int_{B_1}(\ell^-(x))^2\,dx\ge \int_{{\mathcal{B}}}(\ell^-(x))^2\,dx\ge\int_{{\mathcal{B}}}\left(\frac{a}{320\Lambda_\G}\right)^2\,dx\ge \overline{c} a^2,
	\end{eqnarray*}
	for some positive universal constant $\overline{c}$. This establishes the desired contradiction if $\e$ is sufficiently small, and thus the proof of \eqref{lower-bound-linear-funct-1} is complete.
	\medskip
	
	\noindent{\em Step 2.3: conclusion of the proof of \eqref{estimate-zero-set-almost minim}.}
	We can now address the completion of the proof of the measure estimate in \eqref{estimate-zero-set-almost minim}.
	\medskip
	
	Recalling the Poincar\'e-Sobolev inequality (see, e.g., \cites{FLW95,Lu94}),  it holds
	\begin{equation*} 
		\left(\int_{B_1}\left| u(x)-\ell(x)\right|^{2^*}\,dx\right)^{1/2^*}\le C
		\left(\int_{B_1}\left|\Hnabla \big(u(x)-\ell(x)\big)\right|^{2}\,dx\right)^{1/2}\end{equation*}
	for some $C>0$ universal, where
	\[2^*:= \frac{2Q}{Q-2}.\]		
	Then, by virtue of \eqref{second-alternative-dichotomy}, \eqref{defin-linear-funct}, we get 
	\begin{eqnarray*}
		&&\left(\displaystyle\int_{B_1}\left|u(x)-\ell(x)\right|^{2^*}\,dx\right)^{1/2^*}
		\le C	 \left(\int_{B_1}\left|\Hnabla \big(u(x)-\ell(x)\big)\right|^{2}\,dx\right)^{1/2}\\
		&&\qquad \qquad = C \left(\int_{B_1}\left|\Hnabla u(x)-\bq(x)\right|^{2}\,dx\right)^{1/2}
		\le C\varepsilon a.
	\end{eqnarray*}
	This and \eqref{lower-bound-linear-funct-1}  entail
	\begin{align*}
		C\e a\ge \left(\int_{B_{\tau}\cap  \left\{u=0\right\}}\left|\ell(x)
		\right|^{2^*}\,dx\right)^{1/2^*}\ge c_1a\left|B_{\tau}\cap  \left\{u=0\right\}\right|^{1/2^*},
	\end{align*}
	 which means, up to renaming constants,		
	\begin{equation}\label{Sobolev-Poincare-ineq-p-<-n-1}
		\left|B_{\tau}\cap  \left\{u=0\right\}\right|\le C\e^{2^*}.
	\end{equation}
	We notice that
	\[2^*=\frac{2Q+4-4}{Q-2}=2+\frac{4}{Q-2}.\]
	Therefore, setting 
	\begin{equation}\label{casepminoren}\delta:=\frac{4}{Q-2}>0,
	\end{equation}
	 \eqref{estimate-zero-set-almost minim} immediately follows from \eqref{Sobolev-Poincare-ineq-p-<-n-1}.
	
	\medskip

    \noindent{\em Step 3: energy comparison.} By \eqref{second-ineq-condition-almost minim} and \eqref{estimate-zero-set-almost minim}, we have
	\begin{equation}\label{forth-ineq-condition-almost minim}
		\int_{B_{\tau}}\left|\Hnabla u(x)-\Hnabla v(x)\right|^2\,dx\le C_1\varepsilon^{2+\delta}
		+|B_1|\sigma(a^2+1).
	\end{equation}
	Consequently, using \eqref{second-alternative-dichotomy} and \eqref{forth-ineq-condition-almost minim}, it holds
	\begin{equation}\label{stima-L2-gradv-q}\begin{split}
			\int_{B_{\tau}}\left|\Hnabla v(x)-\bq(x)\right|^2\,dx \le\;& 2\left(
			\int_{B_{\tau}}\left|\Hnabla u(x)-\bq(x)\right|^2\,dx
			+\int_{B_{\tau}}\left|\Hnabla v(x)-\Hnabla u(x)\right|^2\,dx\right)\\
			\le\; &  C_2\varepsilon^2a^2+C_1\varepsilon^{2+\delta}+C_3\sigma(a^2+1),
	\end{split}\end{equation}
	up to renaming constants. Let us suppose $\sigma\le c_0\varepsilon^2$, with $c_0$ to be made precise later.
	Recalling \eqref{stima-L2-gradv-q}, we infer
    \begin{equation}\label{first-ineq-p-harm-repl-minus-lin-funct}
		\int_{B_{\tau}}\left|\Hnabla v(x)-\bq(x)\right|^2\,dx\le C_2\varepsilon^2a^2+C_1\varepsilon^{2+\delta}+C_3c_0\varepsilon^2(a^2+1)
		\le C\varepsilon^2a^{2},
	\end{equation}
	for some $C>0$ universal.

	\noindent{\em Step 4: conclusion of the proof of Lemma \ref{lemma-second-alternative-dichotomy-improved}.}	Since, by Remark \ref{rmk-F}, $v-\left \langle \bq,\pi_\cdot \right \rangle$ is $\G$-harmonic in $B_\tau$ and $$\big(u-\left \langle \bq,\pi_\cdot\right\rangle \big)-\big(v-\left \langle \bq,\pi_\cdot\right\rangle\big)= u-v\in HW^{1,2}_0(B_\tau),$$ we have that $v-\left \langle \bq,\pi_\cdot \right \rangle$ is $\G$-harmonic replacement of $u-\left \langle \bq,\pi_\cdot \right \rangle$ in $B_{\tau}$, being $v$ the $\G$-harmonic replacement of $u$ in $B_\tau$. Thus, by Theorem 1.1 in \cite{CM22}, it exists a constant $C_0>0$, depending only on $Q$, such that, for every $x\in B_{\tau}$,
	\begin{equation}\label{CM-1-bis}
		\left|\Hnabla v(x)-\mathbf{q}(x)\right|^2\le \sup_{B_{\tau}}\left|\Hnabla v-\mathbf{q}\right|^2\le C_0\fint_{B_{\tau}}
		\left|\Hnabla v(y)-\bq(x)\right|^2\,dy\le C_0\e^2 a^{2},
	\end{equation}
	where the last inequality is a consequence of hypothesis \eqref{second-alternative-dichotomy}. Consequently, \begin{equation}\label{stima-gradv-q}
		\left|\Hnabla  v- \mathbf{q}\right|\le C\e a\quad\mbox{in }B_\tau
	\end{equation}
	for some constant $C>0$ universal.
	
	Denoting by $\bar{q}_j:=X_jv(0)-q_j$ for $j=1,\ldots,m_1$, let us define the constant horizontal section $$\bar{\bq}(x):=\sum_{j=1}^{m_1}\bar{q}_jX_j(x), \quad \text{for } x\in \G.$$ By \eqref{stima-gradv-q}, for all $x\in B_{\tau}$, we have \begin{equation}\label{stima-bar-q}
		|\bar{\bq}(x)|=|\bar{\bq}(0)|= |\Hnabla v(0)-\mathbf{q}(0)|\le C\varepsilon a.
	\end{equation}
    Hence, combining \eqref{stima-gradv-q} and \eqref{stima-bar-q}, we deduce that, for $x\in B_{\tau}$,
    \begin{equation}\label{stima-aggiunta}
		\left|\Hnabla v(x)-\bq(x)-\bar{\bq}(x)\right|\leq \left|\Hnabla v(x)-\bq(x)\right|+\left|\bar{\bq}(x)\right|\le C\varepsilon a,
	\end{equation}
	up to renaming $C>0$ universal.
	
	By Theorem 1.3  in \cite{CM22}, it exists a constant $C_0>0$, depending only on $Q$, such that, for each $\rho\in(0,\tau/2)$, it holds
	\begin{equation}\label{ineq-average-nabla-p-harm-repl-minus-tilde-q}
		\fint_{B_{\rho}}\left|\Hnabla v(x)-\bq(x)-\bar{\bq}(x)\right|^2\,dx\le 
		\fint_{B_{\rho}}\left(C_0\left(\frac{\rho}{\tau}\right)^\mu
		\left\|\Hnabla v-\bq-\bar{\bq}\right\|_{L^\infty(B_{\tau})}
		\right)^2\,dx\le C_2\,\rho^{2\mu }\,\varepsilon^2 a^{2},
	\end{equation}
	for some $\mu=\mu(Q)\in(0,1]$ and $C_2>0$ universal, where in the last inequality we used estimate \eqref{stima-aggiunta}. Putting together \eqref{forth-ineq-condition-almost minim} and \eqref{ineq-average-nabla-p-harm-repl-minus-tilde-q}, we get
    \begin{equation}\label{ineq-nabla-almost minim-minus-tilde-q}\begin{split}&
			\fint_{B_{\rho}}\left|\Hnabla u(x)-\bq(x)-\bar{\bq}(x)\right|^2\,dx\\
			\le\;& 2\left(
			\fint_{B_{\rho}}\left|\Hnabla u(x)-\Hnabla v(x)\right|^2\,dx+
			\fint_{B_{\rho}}\left|\Hnabla v(x)-\bq(x)-\bar{\bq}(x)\right|^2\,dx
			\right)\\
			\le\;& 2C_1\varepsilon^{2+\delta}
			\rho^{-Q}+2{C}\sigma(a^2+1)\rho^{-Q}+2C_2\rho^{2\mu }\varepsilon^2a^{2},\end{split}
	\end{equation}
	up to renaming constants.
	
	Now, setting $\alpha_0:=\mu$, for every $\alpha \in (0,\alpha_0)$, we choose
	\begin{equation}\label{choose-const}  
		\rho:=\min\left\{(8C_2)^{\frac{1}{2(\alpha- \alpha_0)}},\frac{\tau}{2}\right\}, \qquad
		\e_0:= \left(\frac{\rho^{2\alpha+Q}a_0^2}{{8}C_1}\right)^{\frac1{\delta}},
		\quad{\mbox{and}}\quad
		c_0:=\frac{\rho^{2\alpha +Q} a_0^2}{8{C}(a_1^2+1)}.
	\end{equation}
	With these choices, it holds, for every $\e\leq\e_0$ and $\sigma\leq c_0\e^2$,  
	\begin{eqnarray*}
		&&	2C_2\rho^{2{\mu}} \le \frac{1}{4}\rho^{2\alpha},\\
		&&	2C_1\varepsilon^{2+\delta} \rho^{-Q}\le\frac{1}{4}\rho^{2\alpha }\varepsilon^2 a^2\\
		{\mbox{and }}&& 2 {C}\sigma(a^2+1)\rho^{-Q}\le\frac{1}{4}\rho^{2\alpha }\varepsilon^2 a^2.
	\end{eqnarray*}	
	Consequently, using \eqref{ineq-nabla-almost minim-minus-tilde-q}, we reach
	\begin{align*}
		&\fint_{B_{\rho}}\left|\Hnabla u(x)-\bq(x)-\bar{\bq}(x)\right|^2\,dx\le
		\frac{1}{4}\rho^{2\alpha }\varepsilon^2a^{2 }+\frac{1}{4}\rho^{2\alpha }\varepsilon^2a^2+\frac{1}{4}\rho^{2\alpha }\varepsilon^2a^{2 }\le \rho^{2\alpha }\varepsilon^2a^{2 },
	\end{align*}
	which gives the desired result in \eqref{tesi-lemma-improv} by calling $\widetilde{q}_i:=q_i+\bar{q}_i$, for all $i=1,\ldots,m_1$.
	
	Moreover, from \eqref{stima-bar-q} we have $$ |\bq-\widetilde{\bq}|=|\bar{\bq}|\le C\e a,$$
	which establishes \eqref{so3cer5b56859tj45ivnt45-2}.	This completes the proof of Lemma \ref{lemma-second-alternative-dichotomy-improved}.
\end{proof}


 Next, by iterating Lemma \ref{lemma-second-alternative-dichotomy-improved} we obtain the following
estimates.

\begin{cor}\label{corollary-lemma-second-alternative-dichotomy-improved}
	Let $u$ be an almost minimizer for $J$ in $B_1$ (with constant $\kappa$ and exponent $\beta$) and
	$$ a:=\left(\fint_{B_1}\left| \Hnabla u(x)\right|^2 \,dx\right)^{1/2}.$$
	Suppose that  there exist $a_1>a_0>0$ such that \begin{equation}\label{A0a1}
		a\in[ a_0,a_1]\end{equation} and $u$ satisfies \eqref{bound-q} and \eqref{second-alternative-dichotomy}.	
	
	Then, there exist $\e_0$, $\kappa_0$,  and $\gamma\in(0,1)$, depending on $Q$, $\beta$, $a_0$, and $a_1$, such that, for every $\e\leq\e_0$ and $\kappa\leq\kappa_0\e^2$, it  holds
	\begin{equation}\label{first-conclusion-corollary}
		\left\|u-\ell\right\|_{C^{1,\gamma}(B_{1/2})}\le C\e a
	\end{equation}
	for a positive constant $C$  only depending on  $Q$
	and a $\G$-affine function $\ell$ of slope $\bq$.
	
	Moreover, 
	\begin{equation}\label{second-conclusion-corollary}
		\left\|\Hnabla u\right\|_{L^{\infty}(B_{1/2})}\le \widetilde{C}a,
	\end{equation}
	with $\widetilde{C}>0$  only depending on $Q$.
\end{cor}

\begin{remark}
	We point out that a consequence of \eqref{first-conclusion-corollary}
	in Corollary \ref{corollary-lemma-second-alternative-dichotomy-improved}
	is that, if $\e$ is sufficiently small, 
	\begin{equation}\label{rmk-coroll-hgrad-not-zero}
		\Hnabla u\neq \mathbf{0}\quad {\mbox{in }}B_{1/2},\end{equation}
	where $\mathbf{0}$ is the null horizontal section.
	Indeed, by \eqref{first-conclusion-corollary}, we have 
	\[\left| \Hnabla u-\bq\right| \le C\e a\quad \mbox{in }B_{1/2},
	\]
	which gives
	\[
	\left|\Hnabla u\right| \ge \left|\bq\right|-C\e a\quad \mbox{in }B_{1/2}.
	\]
	Finally, by \eqref{bound-q}, we get
	\[\left|\Hnabla u\right| \ge \frac{a}{4}-C\e a>0\quad  \mbox{in }B_{1/2},\]
	as  long as $\e$  is small enough.
	
	Furthermore, we can conclude that
	\begin{equation}\label{rmk-coroll-assurdo}
		u>0\quad {\mbox{in }} B_{1/2}.\end{equation}
	To check this, we suppose by contradiction that
	there exists a point $x_0\in B_{1/2}$ such that $u(x_0)=0$ (by assumption $u\ge0$).  Therefore, since $u\in C^{1,\gamma}(B_{1/2})$, we know that $\Hnabla u(x_0)=0$,  which contradicts \eqref{rmk-coroll-hgrad-not-zero}, and so \eqref{rmk-coroll-assurdo} holds.
\end{remark}

In order to prove Corollary \ref{corollary-lemma-second-alternative-dichotomy-improved}, we recall the notion of Morrey-Campanato spaces in the setting of Carnot Groups, see \cite{MS79}.

\begin{defin}[Morrey-Campanato spaces]\label{defin-Campanato-spaces}
	Let $\Omega \subseteq \G$. For every $1\le p<+\infty$ and $\lambda \in (0,+\infty)$,  a horizontal section $f=(f_1,\ldots,f_{m_1})\in L^p_{\mathrm{loc}}(\Omega,H\G)$  is said to be in the Morrey-Campanato space $\mathcal{E}^{\lambda,\,p}(\Omega,H\G)$ if
	$$[f]_{\mathcal{E}^{\lambda,\,p}(\Omega,H\G)}:=\sup_{B\subset\Omega}\left(\frac1{|B|^{1+p\lambda}}\int_B|f(y)-\mathbf{f_B}(y)|^p\,dy\right)^{1/p}<+\infty,$$
	where the supremum is taken over all metric balls $B\subset \Omega$ such that $\overline{B}\subset \Omega$ and $\mathbf{f_B}$ is the constant horizontal section given by
	$$\mathbf{f_B}(y):=\sum_{i=1}^{m_1}\left(\fint_Bf_i(z)\,dz\right)X_i(y),\quad \text{for } y\in\Omega.$$
\end{defin}

\begin{remark}
	For every $u=(u_1,\ldots,u_{m_1})\in \mathcal{E}^{\lambda,p}(\Omega,H\G)$, the quantity $[u]_{\mathcal{E}^{p,\lambda}(\Omega,H\G)}$ is a seminorm in $\mathcal{E}^{\lambda,p}(\Omega,H\G)$ and \begin{equation}\label{equivalent-Campanato-seminorm}
		{\mbox{it is equivalent to the quantity }}\quad
		\sup_{B\subset\Omega}\left(\frac1{|B|^{1+p\lambda}}\inf_{\boldsymbol{\xi}}\int_{B}\left|u(x)-\bxi(x)\right|^p\,dx\right)^{1/p},
	\end{equation}
	where the supremum is taken over all metric balls $B$ such that $\overline{B}\subset \Omega$ and the infimum is taken over all the constant horizontal sections such that 
	\begin{equation}\label{b-xi}
		\bxi(x):=\sum_{i=1}^{m_1}\xi_iX_i(x), \quad \text{for }x\in\G \text{ and some }\xi:=(\xi_1,\ldots,\xi_{m_1})\in\R^{m_1}.
	\end{equation}    
	To prove \eqref{equivalent-Campanato-seminorm}, we observe that for every ball $B\subset \Omega$ such that $\overline{B}\subset \Omega$ and constant horizontal section $\bxi$ as in \eqref{b-xi}, by applying  Jensen's Inequality, we have 
	\begin{align*}
		\left|\mathbf{u_B}- \boldsymbol{\xi} \right|^p &\leq\left |\sum_{i=1}^{m_1}\left(\fint_B |u_i(y)-\xi_i|\,dy\right)^2\right |^{\frac p 2}\leq \left | \sum_{i=1}^{m_1} \fint_B|u_i(y)-\xi_i|\,dy \right |^{p}\\
		&\leq C\left(\fint_B \left|u(y)-\boldsymbol{\xi}(y)\right|\,dy\right)^{p}\leq C\fint_B \left|u(y)-\bxi(y)\right|^p\,dy,
	\end{align*}  
	for some constant $C>0$  only depending on $p$ and $\G$.  Consequently, up to renaming $C>0$, it occurs, 
    \begin{align*} 
		\int_B\left | u(y)-\mathbf{u_B}(y) \right |^p\,dy&\leq 2^p\int_B \left | u(y)-\boldsymbol{\xi}(y)\right |^p\,dy\,+\,2^p|B|\left | \mathbf{u_B}-\boldsymbol{\xi} \right |^p\\
		&\leq C \int_B\left | u(y)-\boldsymbol{\xi}(x) \right |^p\,dy
	\end{align*}
        for every constant horizontal section $\boldsymbol{\xi}$ as in \eqref{b-xi} and for every ball $B\subset\Omega$ such that $\overline{B}\subset \Omega$, which leads to $$\int_B \left | u(y)-\mathbf{u_B}(y) \right|^p\, dy\leq C\inf_{\boldsymbol{\xi}}\int_B\left | u(y)-\boldsymbol{\xi}(y) \right |^p\,dy.$$
        On the other hand, for every ball $B\subset\Omega$ such that $\overline{B}\subset \Omega$, it is clear that $$\int_B \left | u(y)-\mathbf{u_B}(y) \right |^p\, dy\geq\inf_{\boldsymbol{\xi}}\int_B\left | u(y)-\boldsymbol{\xi}(y) \right |^p\,dy.$$\end{remark}

Now, we  recall a result, proved in \cite{MS79}, which we will use in the proof  of Corollary \ref{corollary-lemma-second-alternative-dichotomy-improved}.

\begin{thm}\emph{(\cite{MS79}*{Theorem 5})}\label{theo-isomorph-Campanato-spaces-Holder-spaces}
	For every metric ball $B \subset \G$, $\lambda\in (0,1)$, and $p\in[1,+\infty)$, one has $$\mathcal{E}^{\lambda,p}(B,H\G)=C^{0,\lambda}(B, H\G).$$
	More precisely, it results that a function $\phi$ belongs to $\mathcal{E}^{\lambda,p}(B)$ if and only if $\phi$ is equal, $\mathcal{L}^n$-almost everywhere, to a function $\psi$ in the H\"older class $C^{0,\lambda}(B)$. Moreover, the seminorms $[\phi]_{\mathcal{E}^{\lambda,p}(B,H\G)}$ and $[\psi]_{C^{0,\lambda}(B,H\G)}$ are equivalent.
\end{thm}

We also point out the following scaling property
of almost minimizers.

\begin{lem}\label{remark-rescaling-almost minimizer-1}
	Let $u$ be an almost minimizer for $J$ in $B_1$ with constant $\kappa$ and exponent $\beta$. 
	For any $r\in(0,1)$, let
	\begin{equation}\label{defin-rescaling}
		u_r(x):=\frac{u(\delta_r(x))}{r}.
	\end{equation} 
	Then, $u_r$ is			
	an almost minimizer for $J$ in $B_{1/r}$ with constant $\kappa r^\beta$ and exponent $\beta$, namely,
	\begin{equation}\label{THdefin-rescaling} J(u_r,B_\varrho(x_0))\leq(1+\kappa r^\beta \varrho^\beta)J(v,B_\varrho(x_0)),
	\end{equation}
	for every ball $B_\varrho(x_0)$ such that $\overline{B_\varrho(x_0)}\subset B_{1/r}$  and $v\in HW^{1,2}(B_\varrho(x_0))$ such that $u_r-v \in HW^{1,2}_0(B_\varrho(x_0))$.
\end{lem}

\begin{proof} 
	 Given $x_0\in B_{1/r}$, we take $\varrho$ and $v$ as in the statement of Lemma \ref{remark-rescaling-almost minimizer-1} and we define
	\begin{equation}\label{defin-v_r}
		w(x):= rv\left(\delta_{1/r}(x)\right).
	\end{equation}	
	Then, with the notation $y_0:=\delta_r(x_0)$ and $\vartheta:=r\varrho$, for all $x\in\partial B_\vartheta(y_0)$, we  have $\delta_{1/r}(x)\in\partial B_\varrho(x_0)$ and therefore, 
	\begin{equation*}
		w(x)-u(x)=rv\left(\delta_{1/r}(x)\right)-u(x) = rv\left(\delta_{1/r}(x)\right)-ru_r\left(\delta_{1/r}(x)\right)=0.
	\end{equation*}
	Accordingly, we can  exploit $w$ as a competitor for $u$, obtaining, see Definition \ref{def-AM},
	\begin{equation}\label{00PPJnd2}
		\int_{B_{r\varrho}(y_0)}\Big(\left|\Hnabla u(y)\right|^2+\chi_{\left\{u>0\right\}}(y)\Big)\,dy
		\leq(1+\kappa r^\beta\varrho^\beta)
		\int_{B_{r\varrho}(y_0)}\Big(\left|\Hnabla w(y)\right|^2+\chi_{\left\{w>0\right\}}(y)\Big)\,dy.
	\end{equation}
    Furthermore, using, consistently with \eqref{defin-rescaling}, the notation $w_r(x):=\frac{w(\delta_r(x))}{r}$, with the change of variable $x:=\delta_{1/r}(y)$,  we get,  recalling that $y_0:=\delta_r(x_0)$,
	\begin{equation*} \begin{split}&
			\int_{B_{r\varrho}(y_0)}\Big(\left|\Hnabla w(y)\right|^2+\chi_{\left\{w>0\right\}}(y)\Big)\,dy= {r^Q}
			\int_{B_{\varrho}(x_0)}\Big(\left|\Hnabla w(\delta_r(x))\right|^2+\chi_{\left\{w>0\right\}}(\delta_r(x))\Big)\,dx\\
			&\qquad \qquad=r^Q\int_{B_{\varrho}(x_0)}\Big(\left|\Hnabla w_r(x)\right|^2+\chi_{\left\{w_r>0\right\}}(x)\Big)\,dx,
	\end{split}\end{equation*}
	and a similar identity holds true  replacing $w$ and $w_r$ with $u$ and $u_r$. Plugging this information into \eqref{00PPJnd2}, we reach the desired result in \eqref{THdefin-rescaling} by observing that $v=w_r$ from \eqref{defin-v_r}.\end{proof}

\begin{proof}[Proof of Corollary \ref{corollary-lemma-second-alternative-dichotomy-improved}]
	Thanks to Lemma \ref{remark-rescaling-almost minimizer-1}, without loss of generality  we can suppose that
	\begin{equation}\label{riscalamento}
		{\mbox{$u$ is an almost minimizer for $J$ in $B_2$ (with constant $\kappa$ and exponent $\beta$).}}
	\end{equation}
	We divide  the proof into separate steps.
	
	\medskip
	\noindent{\em Step 1: iteration of Lemma \ref{lemma-second-alternative-dichotomy-improved}.}
	We prove that we can iterate Lemma \ref{lemma-second-alternative-dichotomy-improved} indefinitely
	with $\alpha:=\beta/2$. More precisely, we show that, for all $ k\ge 0$, there exists a constant horizontal section $\bq_k:\G \to H\G$, \begin{equation}\label{def-qk}
		\bq_k(x):=\sum_{j=1}^{m_1}q_{k,j}X_j(x), \qquad \text{for some }\ q_k=:(q_{k,1},\ldots,q_{k,m_1})\in\R^{m_1},
	\end{equation}
	such that
	\begin{equation}\label{bounds-qk}\begin{split}&
			|\bq_k|\in\left[ \frac{a}4-\frac{\widetilde C\e a\left(1-\rho^{k\alpha}\right)}{1-\rho^{\alpha}}, \,C_0a+\frac{\widetilde C\e a\left(1-\rho^{k\alpha}\right)}{1-\rho^{\alpha}}\right]
			\quad \text{for all }x\in\G,\\&
			\left(\fint_{B_{\rho^k}}\left| \Hnabla u(x)-\bq_{k}(x)\right|^2\,dx\right)^{1/2}\le\rho^{{k\alpha}}\e a,\\{\mbox{and }}\quad&
			\left(\fint_{B_{\rho^k}}\left|\Hnabla  u(x)\right|^2\,dx\right)^{1/2}\in
			\left[ \frac{|\bq|}2,2|\bq|\right],
		\end{split}
	\end{equation} 
	where $\rho\in(0,1)$, $C_0>0$, and $\widetilde C>0$ are universal constants.
	
	 We prove it by induction. When $k=0$, we  choose $\bq_0:=\bq$. Then, the desired claims follow from \eqref{bound-q} and \eqref{second-alternative-dichotomy}.
	
	Now, we assume that \eqref{bounds-qk} holds for $k$ and we  show it is true for $k+1$. Setting $r:=\rho^k$  and $u_r(\cdot):=\frac{u(\delta_r(\cdot))}{r}$, by  the inductive assumption we have
	$$ \left(\fint_{B_{1}}\left| \Hnabla u_r(x)-\bq_{k}(x)\right|^2\,dx\right)^{1/2}\le r^{\alpha}\e a=\e_k a,$$
	with $\e_k:=r^{\alpha}\e=\rho^{k\alpha}\e$. Notice that the inductive assumption also yields \eqref{estimate-norm-q}, as long as $\e$ is chosen conveniently small. Therefore,  in view of Lemma \ref{remark-rescaling-almost minimizer-1}, we  can apply Lemma \ref{lemma-second-alternative-dichotomy-improved} on the function $u_r$ with $\sigma:=\kappa r^\beta$. We  remark that the  condition $\sigma\le c_0\e^2$ in Lemma \ref{lemma-second-alternative-dichotomy-improved} translates here into $\kappa\le c_0\e^2$, which  is the requirement in the statement of Corollary \ref{corollary-lemma-second-alternative-dichotomy-improved} (by taking $\kappa_0$ there  smaller or equal to $c_0$). Hence, we get from \eqref{tesi-lemma-improv} and \eqref{so3cer5b56859tj45ivnt45-2}
	that there exists a constant horizontal section $\bq_{k+1}$ such that
	\begin{equation}\label{uso-lemma-imp}\begin{split}
			\left(\fint_{B_\rho}\left| \Hnabla u_r(x)-\bq_{k+1}(x)\right|^2\,dx\right)^{1/2}\le\rho^{\alpha}\e_k a\qquad{\mbox{and}}\qquad
			\left|\bq_k-{\bq}_{k+1}\right|\le \widetilde{C}\e_k a.	
	\end{split}\end{equation}
	 Since $\Hnabla u_r=\Hnabla u(\delta_r(x))$, scaling back, we  find, being $\e_k=\rho^{k\alpha}\e$,
	\begin{equation}\label{riscrit-integr-q_k+1}
		\left(\fint_{B_{\rho^{k+1}}}\left| \Hnabla u(x)-\bq_{k+1}(x)\right|^2\,dx\right)^{1/2}\le\rho^\alpha\rho^{{k\alpha}}\e a=
	\rho^{{(k+1)\alpha}}\e a.
	\end{equation}
	Moreover, using again the inductive assumption and \eqref{uso-lemma-imp},
	\begin{eqnarray*} |\bq_{k+1}|&\le& |\bq_k|+|\bq_k-\bq_{k+1}| \le C_0a+\frac{\widetilde C\e a\left(1-\rho^{k\alpha}\right)}{1-\rho^{\alpha}}+\widetilde{C}\rho^{k\alpha}\e a\\
	&=&C_0a+\frac{\widetilde C\e a\left(1-\rho^{k\alpha}\right)+\widetilde{C}\rho^{k\alpha}\e a(1-\rho^{\alpha})}{1-\rho^{\alpha}}=C_0a+\frac{\widetilde C\e a\left(1-\rho^{(k+1)\alpha}\right)}{1-\rho^{\alpha}}
	\end{eqnarray*}
	and, similarly,
	\begin{eqnarray*} |\bq_{k+1}| &\ge& |\bq_k|-|\bq_k-\bq_{k+1}|\ge \frac{a}4-\frac{\widetilde C\e a\left(1-\rho^{k\alpha}\right)}{1-\rho^{\alpha}}-
		\widetilde{C}\rho^{k\alpha}\e a\\
        &=&\frac{a}{4}-\frac{\widetilde C\e a\left(1-\rho^{k\alpha}\right)+\widetilde{C}\rho^{k\alpha}\e a(1-\rho^{\alpha})}{1-\rho^{\alpha}}=\frac{a}{4}-\frac{\widetilde C\e a\left(1-\rho^{(k+1)\alpha}\right)}{1-\rho^{\alpha}}.\end{eqnarray*}
	 Lastly, from \eqref{riscrit-integr-q_k+1}, we also have
	\begin{eqnarray*}&&
		\left|\left(\fint_{B_{\rho^{k+1}}}\left|\Hnabla u(x)\right|^2\,dx\right)^{1/2}-|\bq_{k+1}|\right|\\
        &&\qquad=
		\frac{1}{|B_{\rho^{k+1}}|^{1/2}}
		\left|\left(\int_{B_{\rho^{k+1}}}\left|\Hnabla u(x)\right|^2\,dx\right)^{1/2}-|\bq_{k+1}||B_{\rho^{k+1}}|^{1/2}\right|\\&&\qquad=
		\frac{1}{|B_{\rho^{k+1}}|^{1/2}}
		\left| \left\|\Hnabla u\right\|_{L^2(B_{\rho^{k+1}})}-\left\|\bq_{k+1}\right\|_{L^2(B_{\rho^{k+1}})}\right|\le
		\frac{1}{|B_{\rho^{k+1}}|^{1/2}}
		\left\|\Hnabla u-\bq_{k+1}\right\|_{L^2(B_{\rho^{k+1}})}\\
		&&\qquad=\left(\fint_{B_{\rho^{k+1}}}\left|\Hnabla u(x)-\bq_{k+1}(x)\right|^2\,dx\right)^{1/2}\le\e a,
	\end{eqnarray*}
	which  yields, recalling that $\bq_0=\bq$ and \eqref{bound-q},
	\begin{eqnarray*}&&
		\left|\left(\fint_{B_{\rho^{k+1}}}\left|\Hnabla u(x)\right|^2\,dx\right)^{1/2}-|\bq|\right|=
		\left|\left(\fint_{B_{\rho^{k+1}}}\left|\Hnabla u(x)\right|^2\,dx\right)^{1/2}-|\bq_0|\right|\\&&\qquad\le
		\left|\left(\fint_{B_{\rho^{k+1}}}\left|\Hnabla u(x)\right|^2\,dx\right)^{1/2}-|\bq_{k+1}|\right|+\sum_{j=0}^k\left||\bq_{j+1}|-|\bq_j|\right|\\&&\qquad\le
		\e a+\widetilde{C}a\sum_{j=0}^k\varepsilon_j
		\le \e a+\widetilde{C}\varepsilon a\sum_{j=0}^{+\infty}\rho^{j\alpha}=
		\left(1+\frac{\widetilde{C}}{1-\rho^\alpha}\right)\e a\\&&\qquad\le
		\left(1+\frac{\widetilde{C}}{1-\rho^\alpha}\right)\e |\bq|\le\frac{|\bq|}2.
        \end{eqnarray*}
        These observations conclude the proof of the inductive step and establish \eqref{bounds-qk}.	\medskip

	\noindent{\em Step 2: Morrey-Campanato estimates.}
	We now want to exploit the Morrey-Campanato  estimates of Theorem \ref{theo-isomorph-Campanato-spaces-Holder-spaces}, here applied to the function $\Hnabla u-\bq$, with the following choices 
	\begin{equation}\label{choices-Campanato}
		p=2,\quad B=B_{1/2},\quad \text{and} \quad \lambda=\alpha/Q.
	\end{equation}  To this end, we claim that, for every $B\subset B_{1/2}$ such that $\overline{B}\subset B_{1/2}$, 	
    \begin{equation}\label{primo-claim-campanato}
		\frac{1}{|B|^{1+2\lambda}}\inf_{\boldsymbol{\xi}}\int_{B}\left| \Hnabla u(x)-\bq(x)-\boldsymbol{\xi}(x)\right|^2\,dx\,\le C\e^2 a^2,
	\end{equation}
	where the infimum is taken over all the constant horizontal sections $\bxi:\G\to H\G$ as in \eqref{b-xi}   and $C>0$ is a positive constant. In what follows, $C$ may change from line to line. To prove  it, we distinguish two cases, either $|B|\ge1$ or $|B|\in(0,1)$.\\ If $|B|\ge1$, we use \eqref{second-alternative-dichotomy}
	 to get
	\begin{eqnarray*}&&
		|B|^{-(1+2\lambda)}\inf_{\bxi}\int_{B}\left| \Hnabla u(x)-\bq(x)-\bxi(x)\right|^2\,dx\le
		\int_{B_{1/2}}\left| \Hnabla u(x)-\bq(x)\right|^2\,dx\le |B_1|\,\e^2 a^2,
	\end{eqnarray*}
	which gives  \eqref{primo-claim-campanato}.\\
	If instead $|B|\in(0,1)$, let  $k_0=k_0(B)\in\N$ be such that $B\subseteq B_{\rho^{k_0}}$ and $ B_{\rho^{k}}\subseteq B$ for all $k>k_0$. Then, by \eqref{bounds-qk} we  have, recalling \eqref{choices-Campanato}, choosing $\bxi=\bq_{k_0}-\bq$,
	\begin{eqnarray*}&&
		|B|^{-(1+2\lambda)}\inf_{\bxi}\int_{B}\left|\Hnabla u(x)-\bq(x)-\bxi(x)\right|^2\,dx\le
		|B_{\rho^{k_0+1}}|^{-(1+2\lambda)}\int_{B_{\rho^{k_0}}}\left|\Hnabla u(x)-\bq_{k_0}(x)\right|^2\,dx\\&&\qquad\qquad\qquad\qquad=
		|B_1|\,\rho^{-Q(k_0+1)(1+2\lambda)}\,|B_{\rho^{k_0}}|
		\fint_{B_{\rho^{k_0}}}\left|\Hnabla u(x)-\bq_{k_0}(x)\right|^2\,dx\\&&\qquad\qquad\qquad\qquad
		\leq C|B_1|\,\rho^{-Q(k_0+1)(1+2\lambda)+Qk_0+2k_0\alpha }\,\e^2 a^2= C|B_1|\,\rho^{-Q-2\alpha}\,\e^2 a^2,
	\end{eqnarray*}
	which  implies \eqref{primo-claim-campanato}.\medskip
	
	\noindent{\em Step 4: conclusion of the proof.}
	Since, by \eqref{second-alternative-dichotomy},
	$$ \left\|\Hnabla u-\bq\right\|_{L^2(B_{1/2})}\le C\e a,$$
	and by \eqref{equivalent-Campanato-seminorm} and \eqref{primo-claim-campanato}, we have
	$$ [\Hnabla u-\bq]_{\mathcal{E}^{2,\lambda}(B_{1/2},H\G)}\le C\e a,$$
	 we can apply  Theorem \ref{theo-isomorph-Campanato-spaces-Holder-spaces}, from which it follows that
	\begin{equation}\label{stiama-holder-hgradu-q}
		\left[\Hnabla u-\bq\right]_{C^{0,\lambda}(B_{1/2},H\G)}\le
		C\e a,\end{equation}  with $\lambda=\alpha/Q\in(0,1)$.
	
	Now, we define the $\G$-affine function $\ell:\G\to \R$ as $\ell(x):=u(e)+\left \langle \bq(x),\pi_x(x) \right \rangle$. First, by Proposition \ref{campi-omogenei0}, we have $\Hnabla\ell=\bq$.  Also, for all $x\in B_{1/2}$, let $\delta_x:=  d_c(0,x)$ and $\gamma_x\colon[0,\delta_x]\to B_{1/2}$ a sub-unitary curve  such that
	\[
	\gamma_x(0)=e,\quad \gamma_x (\delta_x)=x,\quad \text{ and }\quad \dot \gamma_x(t)=\sum_{i=1}^{m_1}h_i(t)X_i(\gamma_x(t))\quad \text{ for a.e.\ $t\in [0,\delta_x]$},
	\]
	with $\sum_{j=1}^{m_1}(h_j(t))^2\leq 1$  for a.e.\ $t\in [0,\delta_x]$ (such a $\gamma$ exists thanks to Chow's Theorem, \cite{BLU}*{Theorem 19.1.3}).  Thus, we can write, according to \eqref{stiama-holder-hgradu-q},
	\begin{align*}
		|u(x)- \ell(x)|&=\left |u(x)-u(e)-\left \langle \bq(x),\pi_x(x)\right\rangle-\left\langle\bq(e),\pi_e(e) \right \rangle\right|\\
		&=\left|\int_0^{\delta_x} \frac{d}{dt}\left(u(\gamma_x(t))-\left \langle \bq(\gamma_x(t)),\pi_{\gamma_x(t)}(\gamma_x(t)) \right \rangle\right)\,dt\right | \\
		&\leq C\int_0^{\delta_x}|\nabla_\G u(\gamma_x(t))-\bq(\gamma_x(t))|\,dt \leq C\e a,
	\end{align*}
	which yields, by arbitrariness of $x\in B_{1/2}$, $$\|u-\ell\|_{L^\infty(B_{1/2})}\le C\e a.$$ This and \eqref{stiama-holder-hgradu-q} establish \eqref{first-conclusion-corollary}.
\end{proof}

\section{Lipschitz continuity of almost minimizers and proof
	of Theorem \ref{theor-Lipsch-contin-alm-minim}}\label{sec:lip1}

 We now provide the proof of Theorem \ref{theor-Lipsch-contin-alm-minim}. 

\begin{proof}[Proof of Theorem \ref{theor-Lipsch-contin-alm-minim}]
	 By virtue of Lemma \ref{remark-rescaling-almost minimizer-1}, up  to rescaling,
	we can assume that $u$ is an almost minimizer with constant \begin{equation}\label{1-e2LS}
		\widetilde{\kappa}:=\kappa s^{\beta},\end{equation} which can be made  as small as we wish by  a suitable choice of $s>0$.
	
	Let $\alpha_0\in(0,1]$ be the structural constant  in Lemma \ref{lemma-second-alternative-dichotomy-improved} and define
	$$ \alpha:=\frac12\min\left\{\alpha_0,\frac\beta{2}\right\}.$$
	We also consider $\e_0$  of Proposition \ref{dic} and  fix $\eta\in(0,1)$ and $M\ge1$ as in Proposition \ref{dic} (corresponding here to  the choice $\e:=\e_0/2$).
	Let us  denote by
	\begin{equation}\label{def-a-tau}
		a(\tau):=\bigg(\fint_{B_{\tau}}\left| \Hnabla   u(x)\right|^2\,dx\bigg)^{1/2}.
	\end{equation}
	We  divide the argument into separate steps.			\medskip
	
	\noindent{\em Step 1: estimating the average.}
	We claim that
	\begin{equation}\label{ineq-a(r)-r<1}
		a(r)\le C(M,\eta)(1+a(1))\quad  \mbox{for every } r\in(0,\eta],
	\end{equation}
	for some $C(M,\eta)>0$, possibly depending on $Q$.

	 To prove it, let us consider the set ${\mathcal{K}}\subseteq\N=\{0,1,2,\dots\}$ containing  all $k$'s for which
	\begin{equation}\label{ineq-a(eta^k)}
		a(\eta^k)\le C(\eta)M+2^{-k}a(1),
	\end{equation}
	 with
	\begin{equation}\label{def-C-eta}
		C(\eta):= 2\eta^{-Q/2}. \end{equation}
	We point out that, for $k=0$, \eqref{ineq-a(eta^k)} is obvious, hence
	\begin{equation}\label{k-not-empty}
		0\in{\mathcal{K}}\ne\varnothing.\end{equation}
	We then distinguish two cases, namely, whether \eqref{ineq-a(eta^k)} holds for every $k$ (that is, ${\mathcal{K}}=\N$) or not (i.e., ${\mathcal{K}}\subsetneqq\N$). 
	\medskip
	
	\noindent{\em Step 1.1: the case $\mathcal{K}=\N$}. For every $r\in(0,\eta],$ we  take $k_0\in\N={\mathcal{K}}$ such that $\eta^{k_0+1}<r\le \eta^{k_0}.$  Thus, according to \eqref{def-a-tau} and \eqref{ineq-a(eta^k)}, we  get
	\begin{equation*}\begin{split}
			&a(r)\le \bigg(\frac{1}{|B_1|\,\eta^{(k_0+1)Q} }\int_{B_{\eta^{k_0}}}\left|\Hnabla  u(x)\right|^2\,dx\bigg)^{1/2}=\eta^{-Q/2}a(\eta^{k_0})\le \eta^{-Q/2}(C(\eta)M+a(1))\\
			&\qquad\qquad\le \eta^{-Q/2}\max\left\{ C(\eta)M,1\right\}(1+a(1))\le C(M,\eta)(1+a(1)),
	\end{split}\end{equation*}
	 provided $C(M,\eta)
	\ge\eta^{-Q/2}\max\left\{ C(\eta)M,1\right\}$. The proof of \eqref{ineq-a(r)-r<1} is thereby complete in this case.	\medskip
	
	\noindent{\em Step 1.2: the case ${\mathcal{K}}\subsetneqq\N$.}			
	By \eqref{k-not-empty}, there exists $k_0\in\N$ such that $\{0,\dots,k_0\}\in{\mathcal{K}}$
	and
	\begin{equation}\label{prop-k0+1}
		k_0+1\not\in{\mathcal{K}}.\end{equation} We notice that this implies, recalling \eqref{def-a-tau} and  \eqref{def-C-eta},
	$$\eta^{-Q/2}M< C(\eta)M\le C(\eta)M+2^{-(k_0+1)}a(1)<
	a(\eta^{k_0+1})\le\eta^{-Q/2}a(\eta^{k_0}),$$
	and therefore
	\begin{equation}\label{DALS:iskcd0}a(\eta^{k_0})> M.
	\end{equation}
	Furthermore, \eqref{prop-k0+1}  also gives
	\begin{equation}\label{DALS:iskcd}
		a(\eta^{k_0+1})>C(\eta)M+2^{-(k_0+1)}a(1)\ge\frac{C(\eta)M+2^{-k_0}a(1)}{2}\ge \frac{a(\eta^{k_0})}2.
	\end{equation}
	 We claim that
	\begin{equation}\label{second-alternative-dichotomy-B_eta^k}
		\left(\fint_{B_{\eta^{k_0+1}}}\left| \Hnabla u(x)-\bq(x)\right|^2 \,dx\right)^{1/2}\le \e a(\eta^{k_0}),
	\end{equation}
	for some constant horizontal section $\bq:\G\to H\G$ such that $$\frac{a(\eta^{k_0})}{4}<\left|\bq\right|\le C_0a(\eta^{k_0}),$$  being $C_0$ the constant given by Proposition \ref{dic}.\\ To prove \eqref{second-alternative-dichotomy-B_eta^k}, we apply the  dichotomy of Proposition \ref{dic} rescaled in the ball $B_{\eta^{k_0}}$. Specifically, we obtain  from Proposition \ref{dic}, exploiting \eqref{alt1}, \eqref{alt2}, and \eqref{bound-q}, that \eqref{second-alternative-dichotomy-B_eta^k} holds true,  unless $a(\eta^{k_0+1})\le \frac{ a(\eta^{k_0}) }2,$ but this contradicts \eqref{DALS:iskcd}. This proves \eqref{second-alternative-dichotomy-B_eta^k}.
	
	Now, we  want to apply Corollary \ref{corollary-lemma-second-alternative-dichotomy-improved} rescaled in the ball $B_{\eta^{k_0+1}}$, namely, with $B_1$ replaced by $B_{\eta^{k_0+1}}$ and $a$  by $a(\eta^{k_0+1})$. 
	 For this purpose, we need to  check that the assumptions of Corollary \ref{corollary-lemma-second-alternative-dichotomy-improved} are  verified in this rescaled situation.  First, we note that,  according to \eqref{DALS:iskcd0} and \eqref{DALS:iskcd},
	$$ a(\eta^{k_0+1})\ge\frac{M}{2}.$$
    Also, since $k_0\in{\mathcal{K}}$,
	$$ a(\eta^{k_0+1})\le\eta^{-Q/2}a(\eta^{k_0})\le \eta^{-Q/2}\big(C(\eta)M+2^{-k_0}a(1)
	\big)\le \eta^{-Q/2}\big(C(\eta)M+a(1)\big).$$
	These  two conditions yields \eqref{A0a1} in this rescaled setting with
	\begin{equation}\label{speriamobene}
		a_0:=\frac{M}2\qquad {\mbox{ and }}\qquad a_1:=\eta^{-Q/2}\big(C(\eta)M+a(1)\big).\end{equation}
    Moreover,  from \eqref{DALS:iskcd} and \eqref{second-alternative-dichotomy-B_eta^k},  it holds
	$$\left(\fint_{B_{\eta^{k_0+1}}}\left| \Hnabla u(x)-\bq(x)\right|^2 \,dx\right)^{1/2}\le 2\e a(\eta^{k_0+1}),
	$$  so \eqref{second-alternative-dichotomy} is satisfied (here with $2\e$  in place of $\e$). Finally, we achieve from \eqref{DALS:iskcd} and \eqref{second-alternative-dichotomy-B_eta^k}
	$$			\frac{\eta^{Q/2}a(\eta^{k_0+1})}{4}\le\frac{a(\eta^{k_0})}{4}<\left|\bq\right|\le C_0a(\eta^{k_0})\le2C_0a(\eta^{k_0+1}),
	$$
	which gives that \eqref{bound-q} is  also fulfilled (even if with different structural constants).  To summarize, all the hypotheses of Corollary \ref{corollary-lemma-second-alternative-dichotomy-improved} in this rescaled context are satisfied.
	
	We can  thus exploit an adapted version of Corollary \ref{corollary-lemma-second-alternative-dichotomy-improved}.  This tells us that there exist $\e_0$ (possibly different from the one coming from Proposition \ref{dic}) and $\kappa_0$, depending on $Q$, $\beta$, $a_0$ and $a_1$,
	such that  if
	\begin{equation}\label{condaggforseokroetu}
		\widetilde\kappa\leq\kappa_0\e_0^2,\end{equation} 
	 then we have
	\begin{equation}\label{swqevu3576435432647328528765uhgfhdf}
		\left\|\Hnabla u\right\|_{L^{\infty} (B_{\eta^{k_0+1}/2})}\le \bar{C}a(\eta^{k_0}),
	\end{equation}
	for some  universal constant $\bar{C}>0$. We  stress that  \eqref{condaggforseokroetu} holds by choosing $s$ in \eqref{1-e2LS} small enough, i.e., taking $s:= \left(\frac{\kappa_0 \e_0^2}{2\kappa}\right)^{\frac1{\beta}}$.
	Notice, in particular, that
    \begin{equation}\label{noticeinpartlsworeteuyti8787686}
		\mbox{$s$ depends on $Q$, $\kappa$,
				$\beta$, and $\|\Hnabla u\|_{L^2(B_1)}$},
    \end{equation}  by virtue of \eqref{speriamobene}.

    Using \eqref{swqevu3576435432647328528765uhgfhdf}, we then get, for all $r\in\Big(0,\frac{\eta^{k_0+1}}2 \Big)$, recalling the definition of $k_0$,
	\begin{equation}\label{ineq-a(r)-r-le-eta^k_0/2}\begin{split}&
			a(r)=\left(\frac1{|B_r|}\int_{B_r}|\Hnabla u(x)|^2\,dx\right)^{1/2}\le \bar{C}a(\eta^{k_0})\le
			\bar{C}\big(C(\eta)M+2^{-k_0}a(1)\big)\\&\qquad\qquad\le\bar{C}\big(C(\eta)M+a(1)\big)\le C(M,\eta)(1+a(1)),
	\end{split}\end{equation}
	as long as $C(M,\eta)\ge\bar{C}(C(\eta)M+1)$.\\ Otherwise, if $r\in\left[\frac{\eta^{k_0+1}}2,\eta\right]$, we
take $k_r\in\N$ such that $\eta^{k_r+1}<r\le\eta^{k_r}$. Note that
	$$\frac1{\eta^{k_r}}\le\frac1r\le\frac{2}{\eta^{k_0+1}},$$
	whence
	\begin{equation}\label{L23tSx32}
		k_r \le k_0+C_\star,
	\end{equation}
	where $C_\star:=1+\frac{\log2}{\log(1/\eta)}$.\\    We then distinguish two cases depending on whether $k_r\in\{0,\dots,k_0\}$.  If this holds, then $k_r\in{\mathcal{K}}$ and therefore
	\begin{equation*}
		a(\eta^{k_r})\le C(\eta)M+2^{-k_r}a(1),
	\end{equation*}
	from which, recalling \eqref{ineq-a(eta^k)}, we obtain 
        \begin{equation}\label{kpd012}
		\begin{split}&
			a(r)\le\left(\frac1{|B_{\eta^{k_r+1}}|}\int_{B_{\eta^{k_r}}}|\Hnabla u(x)|^2\,dx\right)^{1/2}
			=\eta^{-Q/2}a(\eta^{k_r})\\&\qquad\qquad\le\eta^{-Q/2}\big(C(\eta)M+2^{-k_r}a(1)\big)\le C(M,\eta)(1+a(1)).\end{split}
	\end{equation}
	If instead $k_r>k_0$, by \eqref{L23tSx32}, we have
	\begin{eqnarray*}&&
		a(r)\le\left(\frac1{|B_{\eta^{k_0+C_\star+1}}|}\int_{B_{\eta^{k_0}}}|\Hnabla u(x)|^2\,dx\right)^{1/2}
		=\eta^{-Q(C_\star+1)/2} a(\eta^{k_0})\\ &&\qquad\qquad\quad\le
		\eta^{-Q(C_\star+1)/2}\big(C(\eta)M+2^{-k_0}a(1)\big)\le\eta^{-Q(C_\star+1)/2}\big(C(\eta)M+a(1)\big)\\&&\quad\qquad\qquad\le
		C(M,\eta)(1+a(1)),
	\end{eqnarray*}
	as long as $C(M,\eta)$ is chosen  sufficiently large. Combining this with \eqref{kpd012} and \eqref{ineq-a(r)-r-le-eta^k_0/2},
	we reach \eqref{ineq-a(r)-r<1}.\medskip
	
	\noindent{\em Step 2: conclusions.}
	 Up to scaling and translations,  we can achieve that \eqref{ineq-a(r)-r<1} is true in all balls with  center $x_0$ in $B_{1/2}$ and sufficiently small radius.  In other words, for all $r\in(0,\eta]$, it holds
	\begin{equation*}
		a(r,x_0)\le C(M,\eta)(1+a(1)),
	\end{equation*}
	where
	$$ 				a(r,x_0):= \left(\fint_{B_{r}(x_0)}\left| \Hnabla u(x)\right|^2\,dx\right)^{1/2}.$$
	 Consequently, in view of the Lebesgue Differentiation Theorem (see, e.g., \cite{T04}), recalling that $u\in  HW^{1,2}(B_1)$, we  obtain, for almost every $x_0\in B_{1/2}$,
	\begin{align*}
		&\left|\Hnabla u(x_0)\right|=\lim_{r\rightarrow 0}a(r,x_0)\le C(M,\eta)(1+a(1))=C\Big(1+\left\|\Hnabla u\right\|_{L^2(B_1)}\Big),
	\end{align*}
	for some $C>0$, and thus
	\begin{equation*} 
		\left\|\Hnabla u\right\|_{L^{\infty}(B_{1/2})}\le C\Big(1+\left\| \Hnabla u\right\|_{L^2(B_1)}\Big).
	\end{equation*}
	
	We now show that the second claim in Theorem \ref{theor-Lipsch-contin-alm-minim} is also true.  To this end, we can suppose that
	\begin{equation}\label{PROCSH}\{u=0\}\cap B_{s/100}\ne\varnothing,\end{equation}
	and we consider a Lebesgue point $\bar{x}\in B_{s/100}$ for $\Hnabla u$.
	Up to a left-translation, we can take that $\bar{x}=0$ and  modify \eqref{PROCSH} into
	\begin{equation}\label{PROCSH2}\{u=0\}\cap B_{s/50}\ne\varnothing.\end{equation}
	We claim that, in this case,
	\begin{equation}\label{alt-c-PK}
		{\mathcal{K}}=\N.
	\end{equation}
	Indeed, suppose this does not hold and let $k_0$ as before (recall \eqref{prop-k0+1}). Hence, according to \eqref{second-alternative-dichotomy-B_eta^k}, we can apply Corollary \ref{corollary-lemma-second-alternative-dichotomy-improved} (rescaled as above) and get,  from \eqref{rmk-coroll-assurdo}, that $u>0$ in $B_{s/2}$, in contradiction  with \eqref{PROCSH2}. This establishes \eqref{alt-c-PK}.\\
	Therefore, in light of \eqref{ineq-a(eta^k)}  and \eqref{alt-c-PK}, for all $k\in\N$,
	$$	a(\eta^k)\le C(\eta)M+2^{-k}a(1),$$
	 which yields
	$$ |\Hnabla u(\bar{x})|=|\Hnabla u(0)|=\lim_{k\to+\infty}a(\eta^k)\le
	\lim_{k\to+\infty}\big(C(\eta)M+2^{-k}a(1)\big)=C(\eta)M.$$
	 The proof of the second claim in Theorem \ref{theor-Lipsch-contin-alm-minim} is then complete by also recalling  \eqref{noticeinpartlsworeteuyti8787686}.
\end{proof}
\noindent F.F. is partially supported by 2023-INDAM-GNAMPA-project {\it Equazioni completamente non lineari locali e non locali.}\\  F.F. is supported by the PRIN research project 2022 7HX33Z - CUP J53D23003610006, "Pattern formation in nonlinear phenomena"\\
N.F. is partially supported by 2023-INDAM-GNAMPA-project {\it Problemi variazionali/nonvariazionali; interazione tra metodi integrali.}\\
E.M. is partially supported by 2023-INDAM-GNAMPA-project {\it Equazioni nonlocali di tipo misto e geometrico.}\\ 
\begin{center}{\bf Data availability statement}
\end{center}Data availability statement Data sharing not applicable to this article as no datasets were generated or analysed
during the current study.
\bibliographystyle{alpha}
 	
\end{document}